%% file: arxiv.tex
\documentclass[a4paper,11pt]{article}
\usepackage{latexsym,amssymb,amsfonts,amsmath,amsthm,mathrsfs,amstext,color,graphicx,times}
\usepackage{stmaryrd}
\usepackage{mathabx}
\usepackage[mathscr]{euscript}
\usepackage{indentfirst}
\usepackage{xcolor}
\usepackage{listings}

\newcommand{\R}{\mathbb{R}}

\newcommand{\N}{\mathbb{N}}

\newcommand{\eps}{\varepsilon}
\DeclareMathOperator{\dom}{dom}
\DeclareMathOperator{\ran}{ran}
\DeclareMathOperator{\gra}{gra}

\newcommand{\tos}{\rightrightarrows} 

\newtheorem{theorem}{Theorem}[section]
\newtheorem{lemma}[theorem]{Lemma}
\newtheorem{proposition}[theorem]{Proposition}
\newtheorem{corollary}[theorem]{Corollary}
\theoremstyle{definition}

\theoremstyle{remark}
\newtheorem{remark}[theorem]{Remark}

\definecolor{mygray}{rgb}{0.95,0.95,0.95}
\lstdefinelanguage{Maxima}{
backgroundcolor=\color{mygray},
keywords={if,then,else,for,thru,in,do,block,sum,makelist,length,addrow,addcol,zeromatrix,ident,augcoefmatrix,ratsubst,diff,ev,tex,sort,listofvars,apply,append,subst,lhs,rhs,minimize_lp,listp,cons,delete,unique,ratsimp,%
with_stdout,nouns,express,depends,load,submatrix,div,grad,curl,%
rootscontract,solve,part,assume,sqrt,integrate,abs,inf,exp},
morekeywords={tuples,ext},
sensitive=true, 
comment=[n][\color{blue}\itshape]{/*}{*/},
basicstyle=\small\ttfamily
}

\newcommand{\ds}{\displaystyle}

\usepackage{wrapfig}
\usepackage[all]{xy}
\usepackage{tikz}
\usepackage{hyperref}
\usepackage{animate}

\usepackage{tikz-3dplot}

\newcommand{\inner}[2]{\langle #1,#2 \rangle}

\DeclareMathOperator{\cone}{cone}
\DeclareMathOperator{\co}{co}

\begin{document}

\title{On the construction of maximal $p$-cyclically monotone operators}


\author{Orestes Bueno
\thanks{Universidad del Pac\'ifico (Lima, Per\'u). Email: \mbox{\texttt{\{o.buenotangoa, cotrina\_je\}@up.edu.pe}}}
\and John Cotrina\footnotemark[1]}
\date{}

\maketitle

\begin{abstract}
In this paper we deal with the construction of explicit examples of maximal $p$-cyclically monotone operators. To date, there is only one instance of an explicit example of a maximal 2-cyclically monotone operator that is not maximal monotone. We present a systematic way to construct this kind of examples, along with several explicit examples.
\medskip

\smallskip

\end{abstract}

\section{Introduction}\label{intro}
Let $U,V$ be non-empty sets. A \emph{multivalued operator} $T:U\tos V$ is an application $T:U\to \mathcal{P}(V)$, that is, for $u\in U$, $T(u)\subset V$. 
The \emph{domain}, \emph{range} and \emph{graph} of $T$ are defined, respectively, as
\begin{gather*}
\dom(T)=\big\{u\in U\::\: T(u)\neq\emptyset\big\},\qquad \ran(T)=\bigcup_{u\in U}T(u)\\ 
\text{and }\gra(T)=\big\{(u,v)\in U\times V\::\: v\in T(u)\big\}.
\end{gather*}
A multivalued operator $T$ is \emph{finite} if its graph is a finite set.
From now on, unless otherwise stated, we will identify multivalued operators with their graphs, so we will write $(u,v)\in T$ instead of $(u,v)\in \gra(T)$. 

Let $V$ be a vector space. If $A\subset V$ then $\co(A)$
denotes the \emph{convex hull}  
of $A$. 
If $T:U\tos V$ is a multivalued operator, the operator $T_{\co}:U\tos V$ is defined as $T_{\co}(u)=\co(T(u))$, for all $u\in U$.

Let $X$ be a Banach space and $X^*$ be its topological dual. 
The \emph{duality product} $\inner{\cdot}{\cdot}:X\times X^*\to\R$ is defined as $\inner{x}{x^*}=x^*(x)$. 
%
Given $C\subset X$ convex, the \emph{normal cone operator} is the operator $N_C:X\tos X^*$ defined as
\[
N_C(x)=\{x^*\in X^*\::\:\inner{y-x}{x^*}\leq 0,\,\forall\,y\in C\},
\]
when $x\in C$ and $N_C(x)=\emptyset$, otherwise. In addition, the \emph{recession cone} of $C$ is the set
\[
0^+C=\{d\in X\::\:\forall x\in C,\,\forall t>0,\, x+td\in C\}.
\]
Given a cone $K\subset X$, the \emph{polar cone} of $K$ is the set 
\[
K^{\circ}=\{x^*\in X^*\::\: \inner{x}{x^*}\leq 0,\,\forall\,x\in K\}. 
\]

%

A multivalued operator $T:X\tos X^*$ is called \emph{monotone} if, for every pair $(x,x^*)$, $(y,y^*)\in T$, 
\[
\inner{x-y}{x^*-y^*}\geq 0.
\]
Moreover, $T$ is \emph{maximal monotone} if $T$ is not properly contained in a monotone operator. 

The notion of $p$-cyclical monotonicity was introduced by Rockafellar in \cite{MR193549} as a ``midpoint'' between classical monotonicity and cyclical monotonicity.  A multivalued operator $T:X\tos X^*$ is called \emph{$p$-cyclically monotone}, with $p\in \N$, if  
\[
\{(x_i,x_i^*)\}_{i=0}^p\subset T\quad\longrightarrow\quad\sum_{i=0}^p\langle x_{i+1}-x_i,x_i^*\rangle\leq 0,
\]
where $x_{p+1}=x_0$. An operator is called \emph{cyclically monotone} if it is $p$-cyclically monotone, for all $p\in\N$.
As in the monotone case, we can consider \emph{maximality} for $p$-cyclically monotone operators: a $p$-cyclically monotone operator is called \emph{maximal $p$-cyclically monotone}, if its graph is not properly contained in the graph of another $p$-cyclically monotone operator.
%

Note that if $T$ is a $p$-cyclically monotone operator, then it is $q$-cyclically monotone, for all $q\leq p$. Moreover, if $T$ is maximal $p$-cyclically monotone and $q$-cyclically monotone, for some $q>p$, then it is also maximal $q$-cyclically monotone. In particular, a maximal monotone operator which is $p$-cyclically monotone, is also maximal $p$-cyclically monotone. Examples of these kinds of operators are the rotation matrices~\cite{Asplund70,Crouzeix2003}. So it is natural to ask if a maximal $p$-cyclically monotone operator is also maximal monotone. The answer was given, in the negative, by Bartz, Bauschke, Borwein, Reich and Wang in~\cite{BBBRW07}.  Bartz \emph{et al.}, using Zorn's Lemma, showed the existence of a maximal 2-cyclically monotone operator that is not maximal monotone. Later on, in~\cite{Bauschke08}, Bauschke and Wang constructed an explicit example of such an operator. This example has also the \emph{bizarre} property of having a non-convex closed domain (in fact, its domain is the boundary of the unit diamond $|x|+|y|=1$.)

Construction of maximal monotone operators was addressed previously by Crouzeix, Oca\~na-Anaya and Sosa~\cite{MR2402763}. See also~\cite{MR3070106}. The goal of this work is to provide a way to construct explicit examples of maximal $p$-cyclically monotone operators. To do this, we use the recently defined \emph{$p$-cyclically monotone polar}~\cite{BueCot19}. We also provide a way to deal with the equations that arise from the definition and the $p$-cyclical monotonicity of this kind of examples. Although the question  whether our method always provides maximal $p$-cyclically monotone operators remains open, we prove maximality of a certain family of examples, namely, operators that can be decomposed as the sum of a normal cone plus a union of perpendicular line segments in $X\times X^*$.

The paper is organized as follows: in section~2, we provide several technical results that are needed later. In section~3, we describe a procedure that constructs explicit examples of $p$-cyclically monotone operators, starting from a finite $p$-cyclically monotone operator. In section~4, we present some computational subroutines that allow us to compute explicit examples and simplify their analysis. In section~5 we present four new examples: three maximal 2-cyclically monotone operators that are not maximal monotone (one with domain in $\R^3$) and a maximal 3-cyclically monotone operator which is not maximal 2-cyclically monotone.

\section{Technical Results}
The $p$-cyclically monotone polar~\cite{BueCot19} is an extension to the $p$-cyclically monotone case of the well known \emph{monotone polar}~\cite{BSML}. The \emph{$p$-cyclically monotone polar} (or simply, $p$-polar) of a multivalued operator $T:X\tos X^*$ is the operator $T^{\mu_p}$ defined via its graph as
\[
T^{\mu_p}=\left\{(x_0,x_0^*)\::\: \sum_{i=0}^{p} \inner{x_{i+1}-x_i}{x_i^*}\leq 0,\,\forall\,\{(x_i,x_i^*)\}_{i=1}^p\subset T\text{ with }x_{p+1}=x_0\right\}.
\]

\begin{proposition}[{\cite{BueCot19}}]\label{pro:pmono}
Let $T,T_i:X\tos X^*$ be multivalued operators and $p\in\N$. The following hold:
\begin{enumerate}
\item $T^{\mu_{p+1}}\subset T^{\mu_p}$.
\item If $T_1\subset T_2$ then $T_2^{\mu_p}\subset T_1^{\mu_p}$.
\item $\displaystyle\left(\bigcup_{i\in I}T_i\right)^{\mu_p}\subset \bigcap_{i\in I}T_i^{\mu_p}$. (Equality holds when $p=1$)
\item The graph of $T^{\mu_p}$ is (strongly-)closed.
 \item $T$ is $p$-cyclically monotone if, and only if, $T\subset T^{\mu_p}$.
 \item $T$ is maximal $p$-cyclically monotone if, and only if, $T=T^{\mu_p}$.
\end{enumerate}
\end{proposition}

\begin{lemma}\label{lem:recc}
Let $T:X\tos X^*$ be a multivalued operator and let $C=\co(\dom(T))$. Then, for all $x_0\in C\cap\dom(T^{\mu_p})$
\[
0^+T^{\mu_p}(x_0)=N_C(x_0).
\]
\end{lemma}
\begin{proof}
Let $n^*\in N_C(x_0)$ and let $x_0^*\in T^{\mu_p}(x_0)$ and $t>0$. Then, for every $\{(x_i,x_i^*)\}_{i=1}^p\subset T$, considering $x_{p+1}=x_0$,
\begin{multline*}
\inner{x_1-x_0}{x_0^*+tn^*}+\sum_{i=1}^p\inner{x_{i+1}-x_i}{x_i^*}\\\leq \inner{x_1-x_0}{x_0^*}+\sum_{i=1}^p\inner{x_{i+1}-x_i}{x_i^*}\leq 0.
\end{multline*}
Hence $n^*\in 0^+T^{\mu_p}(x_0)$.
Conversely, consider $n^*\in 0^+T^{\mu_p}(x_0)$, that is, $x_0^*+tn^*\in T^{\mu_p}(x_0)$, for all $x_0^*\in T^{\mu_p}(x_0)$ and $t>0$. Take $x_1\in \dom(T)$, $x_1^*\in T(x_1)$, and let $x_1=x_2=\cdots=x_p$, $x_1^*=x_2^*=\cdots=x_p^*$, thus
\begin{align*}
0&\geq \inner{x_1-x_0}{x_0^*+tn^*}+\sum_{i=1}^{p-1}\inner{x_{i+1}-x_i}{x_i^*}+\inner{x_0-x_p}{x_p^*}\\
&=\inner{x_1-x_0}{x_0^*+tn^*}+\inner{x_0-x_1}{x_1^*}\\
&=\inner{x_1-x_0}{x_0-x_1^*}+t\inner{x_1-x_0}{n^*}.
\end{align*}
If for some $x_1$, $\inner{x_1-x_0}{n^*}>0$, then taking $t\to+\infty$ in the previous inequality would lead to a contradiction. Therefore $\inner{x_1-x_0}{n^*}\leq 0$, for all $x_1\in \dom(T)$, and this implies that $n^*\in N_C(x_0)$.
\end{proof}

\begin{proposition}\label{pro:copolar}
Let $T:X\tos X^*$ be a multivalued operator, and let $C=\co(\dom(T))$. Then
\[
T^{\mu_p}|_{C}\subset [T_{\co}+N_C]^{\mu_p}\subset [T_{\co}+N]^{\mu_p}\subset  [T_{\co}]^{\mu_p}\subset T^{\mu_p}.
\]
where $N$ is any operator such that $N\subset N_C$. In particular, for every $x\in C$,
\[
[T_{\co}+N_C]^{\mu_p}(x)=[T_{\co}+N]^{\mu_p}(x)=[T_{\co}]^{\mu_p}(x)=T^{\mu_p}(x).
\]
\end{proposition}
\begin{proof}
Since $T\subset T_{\co}\subset T_{\co}+N\subset T_{\co}+N_C$, from Proposition~\ref{pro:pmono}, item 2,
\[
[T_{\co}+N_C]^{\mu_p}\subset [T_{\co}+N]^{\mu_p}\subset [T_{\co}]^{\mu_p}\subset T^{\mu_p}.
\]
Now, given any $(z_0,z_0^*)\in T^{\mu_p}$, with $z_0\in C$, we aim to prove $(z_0,z_0^*)\in (T_{\co}+N_C)^{\mu_p}$. Take $\{(z_j,z_j^*+n_j^*)\}_{j=1}^p\subset (T_{\co}+N_C)$, that is, 
\[
n_j^*\in N_C(z_j)\text{ and }z_j^*=\ds\sum_{l^j=1}^{r_j}\lambda_{j,l^j}x_{j,l^j}^*, \forall\, j=1,\ldots,p,
\]
where $x_{j,l^j}^*\in T(z_j)$, $\lambda_{j,l^j}\geq 0$, for all $l^j=1,\ldots,r_j$, and $\ds\sum_{l^j=1}^{r_j}\lambda_{j,l^j}=1$. Therefore, considering $n_0^*=0$,
\begin{align*}
 \sum_{j=0}^p\inner{z_{j+1}-z_j}{z_j^*+n_j^*}
& \leq \sum_{j=0}^p\inner{z_{j+1}-z_j}{z_j^*}=\sum_{j=0}^p\inner{z_{j+1}-z_j}{\sum_{l^j=1}^{r_j}\lambda_{j,l^j}x_{j,l^j}^*}\\
 &=\sum_{l^0=1}^{r_1}\cdots\sum_{l^p=1}^{r_p}\lambda_{0,l^0}\cdots\lambda_{p,l^p}\sum_{j=0}^p\inner{z_{j+1}-z_j}{x_{j,l^j}^*},
\end{align*}
where the first inequality holds since $z_0\in C$ and $\inner{z_0-z_p}{z_p^*+n_p^*}\leq \inner{z_0-z_p}{z_p^*}$.

Note that the sum $\ds\sum_{j=0}^p\inner{z_{j+1}-z_j}{x_{j,l^j}^*}\leq 0$, for every fixed choice of $l^0,l^1,\ldots,l^p$. Therefore, $T^{\mu_p}|_C\subset [T_{\co}+N_C]^{\mu_p}$ and the first part of the proposition follows. The second part follows immediately.
\end{proof}

\begin{lemma}\label{lem:orto}
Let $\mathcal{F}=\{(w_i,w_i^*)\}_{i=1}^n$ be a $p$-cyclically monotone operator, and consider the operator
\[
S=\bigcup_{i=1}^n\co\{(w_i,w_i^*),(w_{i+1},w_{i+1}^*)\},
\]
where $(w_{n+1},w_{n+1}^*)=(w_1,w_1^*)$. 
If $\inner{w_i-w_{i+1}}{w_i^*-w_{i+1}^*}=0$ then $S$ is $p$-cyclically monotone and $S^{\mu_p}=\mathcal{F}^{\mu_p}$.
\end{lemma}
\begin{proof}
Define $(a_i(t),a_i^*(t))=(1-t)(w_i,w_i^*)+t(w_{i+1},w_{i+1}^*)$, for $t\in[0,1]$ and $i=1,\ldots,n$. Thus, $S$ can be rewritten as
\[
S=\bigcup_{i=1}^n\{(a_i(t),a_i^*(t))\::\:t\in[0,1]\}.
\]
Note that, for fixed $i$ and $t$, since $\inner{w_i-w_{i+1}}{w_i^*-w_{i+1}^*}=0$,
\begin{equation}\label{eq:Baff}
\inner{a_i(t)}{a_i^*(t)}
=\inner{w_i}{w_i^*}+(\inner{w_i}{w_{i+1}^*}-2\inner{w_i}{w_i^*}+\inner{w_{i+1}}{w_i^*})t
\end{equation}
which is affine. 

To prove that $S$ is $p$-cyclically monotone we need to verify that
\[
\sup_{\{(z_j,z_j^*)\}_{j=0}^p\subset S}\sum_{j=0}^p\inner{z_{j+1}-z_j}{z_j^*}\leq 0,
\]
which is equivalent to
\[
\hat\mu=\sup\left\{\sum_{j=0}^p\inner{a_{i_{j+1}}(t_{j+1})-a_{i_j}(t_j)}{a_{i_j}^*(t_j)}\::\:i_j\in\{1,\ldots,n\},t_j\in[0,1]\right\}\leq 0,
\]
Since $S$ is compact, the above supremum is attained, so there exist $\hat i_j\in\{1,\ldots,n\}$, $\hat t_j\in[0,1]$, $j=0,\ldots,p$, such that
\[
\hat\mu=\sum_{j=0}^p\inner{a_{\hat i_{j+1}}(\hat t_{j+1})-a_{\hat i_j}(\hat t_j)}{a_{\hat i_j}^*(\hat t_j)}
\]
Fixing $\hat i_0,\ldots \hat i_p$, because of~\eqref{eq:Baff}, the function
\[
\mu(t_0,\ldots,t_p)=\sum_{j=0}^p\inner{a_{\hat i_{j+1}}(t_{j+1})-a_{\hat i_j}(t_j)}{a_{\hat i_j}^*(t_j)}
\]
is affine on each $t_j$ separately. This implies that in
\begin{align*}
\hat\mu&=\max_{t_0\in[0,1]}\cdots\max_{t_p\in[0,1]}\mu(t_0,\ldots,t_p)\\
&=\max_{t_0\in[0,1]}\cdots\max_{t_p\in[0,1]}\sum_{j=0}^p\inner{a_{\hat i_{j+1}}(t_{j+1})-a_{\hat i_j}(t_j)}{a_{\hat i_j}^*(t_j)}.
\end{align*}
each maximization is attained when either $t_j=0$ or $t_j=1$. Therefore $(a_{\hat i_j}(\hat t_j),a_{\hat i_j}^*(\hat t_j))\in \mathcal{F}$, for all $j=0,\ldots, p$ and thus, $\hat\mu\leq 0$, since $\mathcal{F}$ is $p$-cyclically monotone by hypothesis.

Let $(z_0,z_0^*)\in \mathcal{F}^{\mu_p}$ and define
\begin{multline*}
\tilde\mu=\sup_{\{(z_j,z_j^*)\}_{j=1}^p\subset S}\sum_{j=0}^p\inner{z_{j+1}-z_j}{z_j^*}\\
=\sup\bigg\{\inner{a_{i_1}(t_1)-z_0}{z_0^*}+
\sum_{j=1}^{p-1}\inner{a_{i_{j+1}}(t_{j+1})-a_{i_j}(t_j)}{a_{i_j}^*(t_j)}\\
+\inner{z_0-a_{i_p}(t_p)}{a_{i_p}^*(t_p)}\::\:i_j\in\{1,\ldots,n\},t_j\in[0,1]\bigg\}
\end{multline*}
In the same way as before, the previous supremum is attained, and the values of $t_j$ which maximize such expresion are either 0 or 1. Therefore, for some $\{(y_i,y_i^*)\}_{i=1}^p\subset \mathcal{F}$, 
\[
\tilde\mu=\sum_{j=0}^p\inner{y_{j+1}-y_j}{y_j^*},
\]
where $(y_0,y_0^*)=(z_0,z_0^*)$. Since $(z_0,z_0^*)\in \mathcal{F}^{\mu_p}$, we obtain $\tilde\mu\leq 0$, and the lemma follows.
\end{proof}

\begin{lemma}\label{lem:algop}
Let $S$ be a $p$-cyclically monotone operator, $a\in\dom(S^{\mu_p})$ and consider $T=(\{a\}\times S^{\mu_p}(a))\cup S$. Then, $T$ is also $p$-cyclically monotone.
\end{lemma}
\begin{proof}
First, we will prove that $T$ is monotone. Indeed, given $(b,b^*),(c,c^*)\in T$, let $M=\inner{b-c}{b^*-c^*}$,
\begin{itemize}
 \item if both $b=c=a$ then $M=0$,
 \item if $b=a$ and $c\neq a$, then $M\geq 0$, since $(c,c^*)\in S$ and $(a,b^*)\in S^{\mu_p}$,
 \item if $b,c\neq a$ then $M\geq 0$, since $S$ is $p$-cyclically monotone.
\end{itemize}
Now let $1<r\leq p$. Assume that $T$ is $j$-cyclically monotone, for every $1\leq j< r$. We aim to prove that $T$ is also $r$-cyclically monotone.

Take $\{(z_i,z_i^*)\}_{i=0}^r\subset T$. If $\{(z_i,z_i^*)\}_{i=0}^r\subset S$, then we are done. Otherwise, without loss of generality we may assume that $z_0=a$, so $z_0^*\in S^{\mu_p}(a)$. If the remaining $z_i$, $i=1,\ldots,r$ are different from $a$ then 
\[
\sum_{i=0}^r\inner{z_{i+1}-z_i}{z_i^*}\leq 0
\]
since $(z_0,z_0^*)\in S^{\mu_p}\subset S^{\mu_r}$. Otherwise, there exists $k\neq 0$ such that $z_{k}=a$, so the sum 
\[
\sum_{i=0}^{k-1}\inner{z_{i+1}-z_i}{z_i^*}+\sum_{i=k}^{r}\inner{z_{i+1}-z_i}{z_i^*}\leq 0
\]
since both sums are associated to the finite sets $\{(z_i,z_i^*)\}_{i=0}^{k-1}$ and $\{(z_i,z_i^*)\}_{i=k}^{p}$ and $T$ is already assumed to be $(k-1)$ and $(r-k)$-cyclically monotone.
\end{proof}

The following results deal with finite operators. 
\begin{proposition}\label{pro:opfinite}
Let $\mathcal{F}:X\tos X^*$ be a finite multivalued operator. Given $p\geq 2$ and $z_0\in\dom(\mathcal{F}^{\mu_p})$, the set $\mathcal{F}^{\mu_p}(z_0)$ is a polyhedron and, given $z_0^*\in \mathcal{F}^{\mu_p}(z_0)$, it is defined by either 
\begin{enumerate}
\item at most $\#\dom(\mathcal{F})\cdot\#\ran(\mathcal{F})$ inequalities, each one linear in terms of either $z_0$ or $z_0^*$; or
\item at most $\#\dom(\mathcal{F})$ inequalities, each one linear in terms of $z_0^*$. 
\end{enumerate}
\end{proposition}
\begin{proof}
By definition, $z_0^*\in \mathcal{F}^{\mu_p}(z_0)$ if, and only if,
\begin{equation}\label{eq:polarf}
\sum_{i=0}^p\inner{z_{i+1}-z_i}{z_i^*}\leq 0,\,\forall\{(z_i,z_i^*)\}_{i=1}^p\subset \mathcal{F},
\end{equation}
considering $z_{p+1}=z_0\in C$.  This implies that $\mathcal{F}^{\mu_p}(z_0)$ is a polyhedron, as it is a finite intersection of half-spaces.
Let $Z=\{(z_i,z_i^*)\}_{i=1}^p\subset \mathcal{F}$. Note that the inequality in~\eqref{eq:polarf} is equivalent to
\begin{equation}\label{eq:polarf2}
N(Z)\leq \inner{z_0-z_1}{z_0^*-z_p^*}.
\end{equation}
where $N(Z)=\ds\sum_{i=1}^{p-1}\inner{z_{i+1}-z_i}{z_i^*}+\inner{z_1-z_p}{z_p^*}$.
Denote $\widetilde{N}(z_1,z_p^*)$ as the maximum over all the points $(z_2,z_2^*),\ldots,(z_{p-1},z_{p-1}^*)\in\mathcal{F}$ and all $z_1^*\in \mathcal{F}(z_1)$, $z_p\in \mathcal{F}^{-1}(z_p^*)$. Thus, 
\begin{equation}\label{eq:ntilde}
 \widetilde{N}(z_1,z_p^*):=\max\left\{N(Z)\::\:\begin{array}{c}\{(z_i,z_i^*)\}_{i=2}^{p-1}\subset\mathcal{F}\\ z_1^*\in \mathcal{F}(z_1),\,z_p\in \mathcal{F}^{-1}(z_p^*)\end{array}\right\}
\end{equation}
and from \eqref{eq:polarf2} we deduce
\[
\widetilde{N}(z_1,z_p^*)\leq \inner{z_0-z_1}{z_0^*-z_p^*}.
\]
This is equivalent to
\begin{equation}\label{eq:eqred}
\inner{z_0-z_1}{z_0^*}\geq \widetilde{N}(z_1,z_p^*)+\inner{z_0-z_1}{z_p^*}.
\end{equation}
Note that the last inequality only depends on $z_0$, $z_1$ and $z_p^*$. Therefore, for fixed $z_0$, any inequality generated by an arbitrary choice of $\{(z_i,z_i^*)\}_{i=1}^p$ would be exactly as, or implied by, inequality~\eqref{eq:eqred}. Item {\it 1} of the proposition follows by observing that there are at most $\#\dom(\mathcal{F})$ choices of $z_1$ and $\#\ran(\mathcal{F})$ choices of $z_p^*$. Note that inequality~\eqref{eq:eqred} is linear in terms of both $z_0$ and $z_0^*$ separately. 

Finally, in the same way as before, let $\widetilde{M}(z_0,z_1)$ be the maximum of~\eqref{eq:eqred} over all possible $z_p^*\in\ran(\mathcal{F})$, that is, 
\begin{equation}\label{eq:mtilde}
 \widetilde{M}(z_0,z_1)=\max\{\widetilde{N}(z_1,z_p^*)+\inner{z_0-z_1}{z_p^*}\::\: z_p^*\in \ran(\mathcal{F})\}.
\end{equation}
Therefore~\eqref{eq:eqred} is exactly as, or implied by, 
\begin{equation}\label{eq:eqred2}
\widetilde{M}(z_0,z_1)\leq \inner{z_0-z_1}{z_0^*}.
\end{equation}
Item {\it 2} follows by observing that there are at most $\#\dom(\mathcal{F})$ choices for $z_1$ in~\eqref{eq:eqred2} and that~\eqref{eq:eqred2} is linear in terms of $z_0^*$.
\end{proof}
\begin{remark}
When the elements of $\mathcal{F}$ have small integer components, the quantities $\widetilde{N}(z_1,z_p^*)$ and $\widetilde{M}(z_0,z_1)$ can be exactly computed using any programming language. See Section~\ref{sec:tools} for such an implementation. 
\end{remark}
%
%
%
%
\begin{lemma}[Farkas]
Let $A$ be a $m\times n$ matrix and $b\in\R^n$. The system $Ax\geq b$ if, and only if, 
\[
\forall\, y\geq 0,\, A^{\top}y=0 \Longrightarrow \inner{b}{y}\leq 0.
\]
\end{lemma}

\begin{corollary}\label{cor:mtilde}
Let $\mathcal{F}$ and $\widetilde{M}$ defined as in the Proposition~\ref{pro:opfinite}, and assume $\dom(\mathcal{F})=\{x_1,\ldots,x_n\}$. Then $z_0\in\dom(\mathcal{F}^{\mu_p})$ if, and only if, for all $\lambda_1,\ldots,\lambda_n\geq 0$, such that $\ds\sum_{i=1}^n\lambda_i=1$ and $z_0=\ds\sum_{i=1}^n\lambda_ix_i$,
\[
\sum_{i=1}^n\lambda_i\widetilde{M}(z_0,x_i)\leq 0.
\]
\end{corollary}
\begin{proof}
Let $z_0\in X$. The system of inequalities~\eqref{eq:eqred2} can be written in the form $A(z_0)\cdot z_0^*\geq b(z_0)$, where 
\[
A(z_0)=[z_0-x_1,\ldots,z_0-x_n]^{\top},\qquad b(z_0)=(\widetilde{M}(z_0,x_1),\ldots,\widetilde{M}(z_0,x_n)). 
\]
Now let $y=(\lambda_1,\ldots,\lambda_n)$, with $\lambda_i\geq 0$, but not all zero, then
\[
A(z_0)^{\top}y=\left(\sum_{i=1}^n\lambda_i\right) z_0-\sum_{i=1}^n\lambda_ix_i=\sigma\left(z_0-\sum_{i=1}^n\dfrac{\lambda_i}{\sigma}x_i\right)
\]
where $\sigma=\ds\sum_{i=1}^n\lambda_i>0$. We thus have proved that $A(z_0)^{\top}y=0$ if, and only if, $z_0\in \co(\dom(\mathcal{F}))$ and the weighted components of $y$ are the coefficients associated to $z_0$ as a convex combination of $x_i$'s. The corollary follows by using Farkas' Lemma on system~\eqref{eq:eqred2} and observing that 
\[
\inner{b(z_0)}{y}=\dfrac{1}{\sigma}\sum_{i=1}^n\lambda_i\widetilde{M}(z_0,x_i).
\]
\end{proof}
\begin{corollary}
Let $\mathcal{F}$ be defined as in Proposition~\ref{pro:opfinite}. Then 
\[
X\setminus\co(\dom(\mathcal{F}))\subset \dom(\mathcal{F}^{\mu_p}).
\]
\end{corollary}

\section{The algorithm}\label{sec:algorithm}
From now on, $X$ will denote a real Hilbert space and we will identify $X^*$ with $X$. Let $T_0=\{(x_i,x_i^*)\}_{i=1}^n$ be a $p$-cyclically monotone operator. Consider, for $k=1,\ldots,n$, the following sequence of operators:
\begin{align*}
T_k=\big(\{x_k\}\times T_{k-1}^{\mu_p}(x_k)\big)\cup T_{k-1}.
\end{align*}
and $T_{n+1}=T_n^{\mu_p}$. Note that, for $k=1,\ldots,n$,
\[
T_k=\bigcup_{j=1}^k\big(\{x_j\}\times T_{j-1}^{\mu_p}(x_j)\big)\cup \big\{(x_j,x_j^*)\big\}_{j=k+1}^n.
\]
Lemma~\ref{lem:algop} implies that $T_k$, for each $k=1,\ldots,n$, is $p$-cyclically monotone.

From now on, consider $C=\co(\{x_1,\ldots,x_n\})$. Note that $C$ is convex and closed.

\begin{proposition}\label{pro:proptnp1}
The following hold:
\begin{enumerate}
\item $T_1,\ldots,T_n$ are $p$-cyclically monotone.
\item For each $k=1,\ldots,n$, $T_{k-1}^{\mu_p}(x_k)$ is a polyhedron, so it can be written as
\[
 T_{k-1}^{\mu_p}(x_k)=\co(E_k)+N_C(x_k)
\]
for some finite set $E_k$.
\item Let $\mathcal{F}=\ds\bigcup_{k=1}^n \{x_k\}\times E_k$. Then 
\begin{equation}\label{eq:tn}
T_n=\bigcup_{k=1}^n \{x_k\}\times(\co(E_k)+N_C(x_k))=\mathcal{F}_{\co}+N_C.
\end{equation}
\item $\dom(T_{n+1})\subset C$.
\item $T_{n+1}(x_k)=T_n(x_k)$, for all $k=1,\ldots,n$.
\item $T_{n+1}=\ds\mathcal{F}^{\mu_p}\big|_C$
\end{enumerate}
\end{proposition}
\begin{proof}
\begin{enumerate}
\item It is an immediate consequence of Lemma~\ref{lem:algop}.
\item The set $T_0^{\mu_p}(x_1)$ is a polyhedron since $T_0$ is a finite set and by Proposition~\ref{pro:opfinite}. Let $k\in\{2,\ldots,n\}$ and assume that $T_{j-1}^{\mu_p}(x_j)$ is a polyhedron, for all $j<k$. Let $E_j\subset T_{j-1}^{\mu_p}(x_j)$ be a finite set such that $T_{j-1}^{\mu_p}(x_j)=\co(E_j)+N(x_j)$, where, in view of Lemma~\ref{lem:recc}, 
\[
N(x_j)=N_C(x_j)=0^+[T_{j-1}^{\mu_p}(x_j)],\text{ for }j=1,\ldots,k-1, 
\]
and consider $E_j=\{x_j^*\}$ and $N(x_j)=\{0\}$, for $j=k,\ldots,n$. Thus we can write
\[
T_{k-1}=\bigcup_{j=1}^{n}\{x_j\}\times(\co(E_j)+N(x_j)),
\]
that is $T_{k-1}=F_{\co}+N$, where $F=\ds\bigcup_{j=1}^n\{x_j\}\times E_j$ is a finite set and $N\subset N_C$. By Proposition~\ref{pro:copolar}, 
\[
T_{k-1}^{\mu_p}(x_k)=F^{\mu_p}(x_k),
\]
which is a polyhedron, since $F$ is finite.  
\item It follows from item {\it 2}.
 \item Let $x\in\dom(T_{n+1})$ and let $x^*\in T_{n+1}(x)=T_n^{\mu_p}(x)\subset T_n^{\mu_1}(x)$. Therefore, for every $(y,y^*)$ in $T_n$,
\[
\inner{x-y}{x^*-y^*}\geq 0.
\]
In particular, for any $j=1,\ldots,n$, take $t>0$, $x_j^*\in T_n(x_j)$ and $n_j^*\in 0^+T_n(x_j)=0^+T^{\mu_p}_{j-1}(x_j)=N_C(x_j)$, thus $x_j^*+tn_j^*\in T_n(x_j)$, and
\[
\inner{x-x_j}{x^*-(x_j^*+tn_j^*)}\geq 0\quad\iff\quad\inner{x-x_j}{x^*-x_j^*}\geq t\inner{x-x_j}{n_j^*},
\]
for all $t>0$. This implies that $\inner{x-x_j}{n_j^*}\leq 0$, for all $n_j^*\in N_C(x_j)$, that is $x-x_j\in N_C(x_j)^{\circ}=T_C(x_j)$. Therefore, $x\in x_j+T_C(x_j)$, for all $j=1,\ldots,n$. Using Theorem 2.15 in \cite{Bauschke08}, we conclude that
\[
x\in \bigcap_{j=1}^n(x_j+T_C(x_j))=C.
\]
\item Since $T_n$ is $p$-cyclically monotone and $T_{n+1}=T_n^{\mu_p}$, then $T_n(x_k)\subset T_{n+1}(x_k)$. On the other hand, since $T_{k-1}\subset T_n$,
\[
T_{n+1}(x_k)=T_n^{\mu_p}(x_k)\subset T_{k-1}^{\mu_p}(x_k)=T_k(x_k)=T_n(x_k).
\]
\item This is a consequence of equation~\eqref{eq:tn} and Proposition~\ref{pro:copolar}.
\end{enumerate}

\end{proof}
\begin{remark}
By item~{\it 6} in the previous proposition, given $x\in C$, $T_{n+1}(x)$ is a polyhedron. Moreover, Proposition~\ref{pro:opfinite} states that $T_{n+1}(x)$ can be defined by at most $n$ linear inequalities, when $p\geq 2$, or by at most $\#\mathcal{F}$ inequalities, when $p=1$.
\end{remark}

Let $s=(1,\ldots,1)\in\R^n$ and let $S_n$ be the simplex $S_n=\{\lambda\in\R^n\::\:\lambda\geq 0,\inner{\lambda}{s}=1\}$.
\begin{lemma}\label{lem:open}
 Let $x_1,\ldots,x_n\in X$, $C=\co(\{x_1,\ldots,x_n\})$ and let $F:S_n\to C$ be  defined as $F(\lambda)=\ds\sum_{i=1}^n\lambda_ix_i$.
 Then $F$ is an open map, that is, if $U\subset S_n$ is open in $S_n$, then $F(U)$ is open in $C$. 
\end{lemma}
\begin{proof}
The points in $X\times \R^n$:
\[
(x_1,e_1),\ldots,(x_n,e_{n})
\]
are affinely independent, where $e_1,\ldots,e_{n}$ are the canonical vectors in $\R^{n}$. Thus, $\hat F:S_n\to \co\{(x_i,e_i)\}$ is a bijective continuous function between compact sets, hence an homeomorphism. The lemma follows by observing that the projection $\Pi_1:\co\{(x_i,e_i)\}\to C$, $\Pi_1(x,y)=x$, is open and that $F=\Pi_1\circ \hat F$. 
\end{proof}
\begin{proposition}
$\dom(T_{n+1})$ is closed.
\end{proposition}
\begin{proof}
 Let $z_0\in C$, $z_0\notin\dom(T_{n+1})$, and let $\mathcal{F}$ be as in Proposition~\ref{pro:proptnp1}, item 3. Using Corollary~\ref{cor:mtilde} applied to $\mathcal{F}$, there exists $\lambda_0\in S_n$ such that $z_0\in F(\lambda_0)$ and
\[
\widetilde{\mathcal{M}}(\lambda_0)=\sum_{i=1}^n\lambda_i\widetilde{M}(F(\lambda_0),x_i)>0.
\]
Since $\widetilde{\mathcal{M}}$ is continuous, there exists $\delta>0$ such that, for all $\lambda\in B(\lambda_0,\delta)\cap S_n$, $\widetilde{\mathcal{M}}(\lambda)>0$. On the other hand, by Lemma~\ref{lem:open}, $F$ is an open map, so there exists $\eps>0$ such that $B(z_0,\eps)\cap C\subset F(B(\lambda_0,\delta)\cap S_n)$. This implies that, for all $z\in B(z_0,\eps)\cap C$, there exists $\lambda\in B(\lambda_0,\delta)\cap S_n$ such that $z=F(\lambda)$ and, thus, $\widetilde{\mathcal{M}}(\lambda)>0$, that is $z\notin\dom(T_{n+1})$. The proposition then follows.
\end{proof}

\section{Computational Tools}\label{sec:tools}
In this section we present some of the code we used to obtain the explicit formulas of operators $T_n$ and $T_{n+1}$. We used \emph{Maxima}, version 5.41.0, which is an open source symbolic algebra program.

First, we need some preliminary subroutines.
\begin{lstlisting}
/* xc : extracts the "x" component of pts=[x,xs] */
xc(pts) := part(pts,makelist(i, i, 1, length(pts)/2))$
/* xsc : extracts the "xs" component of pts=[x,xs] */
xsc(pts) :=
  part(pts,makelist(i, i, length(pts)/2 + 1 , length(pts)))$

/* xl : extracts the "x" component of a list 
        ptsl=[[x1,xs1],...,[xn,xsn]] */
xl(ptsl) := makelist(xc(ptsl[i]),i,1,length(ptsl))$
/* xsl : extracts the "xs" component of a list
         ptsl=[[x1,xs1],...,[xn,xsn]] */
xsl(ptsl) := makelist(xsc(ptsl[i]),i,1,length(ptsl))$
\end{lstlisting}
The procedure \texttt{tuples} creates a list of all $N$-tuples of elements in the list $S$. This procedure was extracted from~\cite{cite:tuples}.
\begin{lstlisting}
tuples(S, N):=block([K,counter,i,res,j],
   K:length(S),
   counter:create_list(1, i, 1, N),
   res:[create_list(S[counter[i]],i,1,N)],
   for i:1 thru K^N-1 do (
      counter[1]:counter[1]+1,
      for j:1 thru N-1 do (
         if counter[j]>K then (
            counter[j]:1,
            counter[j+1]:counter[j+1]+1
         )
         else(
            return
         )
      ),
      res:append(res, [create_list(S[counter[i]],i,1,N)])
   ),
   res
)$
\end{lstlisting}

\subsection{Testing for $p$-cyclical monotonicity}\label{sec:testpmono}
The procedure \texttt{sump} computes the cyclic sum 
\[
\sum_{k=1}^{r-1}\inner{x_{k+1}-x_{k}}{x^*_{k}} +\inner{x_{1}-x_{r}}{x^*_{r}}
\]
of a list of points $[[x_{1},x^*_{1}],\ldots,[x_{r},x^*_{r}]]$.
\begin{lstlisting}
sump(cyc):=block([x,xs,r],
/*
  cyc : list of points (x,x^*) to be sumed
     must be in the form [[x[1],xs[1]],...,[x[r],xs[r]]]
*/
  x:xl(cyc), /*  x = [x[1],x[2],...,x[r]] */
  xs:xsl(cyc), /* xs = [xs[1],xs[2],...,xs[r]] */
  r:length(cyc), /* length of the cycle */
  /* output */
  sum( (x[k+1]-x[k]).xs[k] ,k,1,r-1)+(x[1]-x[r]).xs[r]
)$
\end{lstlisting}

The procedure \texttt{ispmono} computes the maximum of the sums of all the $p$-cycles of a finite operator $\mathcal{F}$. In this way, it determines if $\mathcal{F}$ is a $p$-cyclically monotone operator.
\begin{lstlisting}
ispmono(opF,p):=block([cyclist,cyc,maxsum,out],
/*
  opF : list of points of operator \mathcal{F}
     must be in the form [[x1,xs1],...,[xn,xsn]]
*/
  cyclist:tuples(opF,p+1), /* list of all p-cycles in F */
  maxsum:-inf,
  for cyc in cyclist do(
     maxsum:max(maxsum,sump(cyc))
  ),
if maxsum<=0 then out:true else out:false,
[out,maxsum] /* output */
)$
\end{lstlisting}
\begin{remark}
Since \emph{Maxima} supports symbolic arithmetic, the previous procedure not only provides an exact output when the input operator $\mathcal{F}$ has integer components, but also when $\mathcal{F}$ contains fractions, integer roots and certain constants (e.g., $e$ and $\pi$).
\end{remark}

\subsection{$\widetilde{N}$ and $\widetilde{M}$ implementations}
To implement $\widetilde{N}(z_1,z_p^*)$ we first need two additional subroutines: given a finite operator $\mathcal{F}$, the procedures \texttt{image} and \texttt{imageinv} compute $\mathcal{F}(z_1)$ and $\mathcal{F}^{-1}(z_p^*)$, respectively, for given $z_1\in\dom(\mathcal{F})$ and $z_p^*\in \ran(\mathcal{F})$. 
\begin{lstlisting}
image(opF,z1):=block([z1slist,i],
/*
opF : list of points of operator \mathcal{F}
    must be in the form [[x1,xs1],...,[xn,xsn]]
z1 : point in dom(F)
*/
   z1slist:[],
   for i:1 thru length(opF) do (
      if z1=xc(opF[i]) then
         z1slist:append(z1slist,[xsc(opF[i])])
   ),
   z1slist /* output F(z1) as a list*/
)$

imageinv(opF,zps):=block([z1slist,i],
/*
opF : list of points of operator \mathcal{F}
    must be in the form [[x1,xs1],...,[xn,xsn]]
zps : point in ran(F)
*/
   zplist:[],
   for i:1 thru length(opF) do (
      if zps=xsc(opF[i]) then
         zplist:append(zplist,[xc(opF[i])])
   ),
   zplist /* output F^{-1}(zps) as a list */
)$ 
\end{lstlisting}

We now present implementations of $\widetilde{N}(z_1,z_p^*)$ and $\widetilde{M}(z_0,z_1)$, as they appear in equations~\eqref{eq:ntilde} and \eqref{eq:mtilde}, respectively.
\begin{lstlisting}
Ntilde(opF,z1,zps,p):=block(
[z1slist,zplist,z1s,zp,pcyc,partpcyc],
   z1slist:image(opF,z1),
   zplist:imageinv(opF,zps),
   Zp:tuples(opF,p-2),
   maxsum:-inf,
   for z1s in z1slist do(
      for zp in zplist do(
         for partpcyc in Zp do(
            pcyc:append( [flatten([z1,z1s])],
                         partpcyc,
                         [flatten([zp,zps])] ),    
            maxsum:max(maxsum,sump(pcyc))
         )      
      )
   ),
   factor(ratsimp(maxsum)) /* output (simplified) */
)$
\end{lstlisting}

\begin{lstlisting}
Mtilde(opF,z0,z1,p):=block([ranF,maxsum,zps],
   ranF:unique(xsl(opF)),
   maxsum:-inf,
   for zps in ranF do(
      maxsum:max(maxsum,Ntilde(opF,z1,zps,p)+(z0-z1).zps)
   ),
   factor(ratsimp(maxsum)) /* output (simplified)*/
)$
\end{lstlisting}

\subsection{Using \emph{Maxima} symbolic algebra features}\label{sec:maxima}
The subroutines implemented in this section make use of \emph{Maxima}'s symbolic features. 

The procedure \texttt{polareqs} computes the system of inequalities in~\eqref{eq:eqred2}. This procedure accepts a finite operator $\mathcal{F}$ and a point $z_0\in\R^d$ as inputs, and returns the list of inequalities that define $\mathcal{F}^{\mu_p}(z_0)$, when $p\geq 2$. The inequalities appear in terms of undetermined variables of the form $Y[1],\ldots,Y[d]$. Note that the input point $z_0$ can also be undetermined, when this is the case, \texttt{polareqs} returns the full system of inequalities that satisfy any $(z_0,z_0^*)\in \mathcal{F}^{\mu_p}$.

\begin{lstlisting}
polareqs(opF,z0,p):=block([eqs,domF,z0s,z1],
/*
opF : list of points of operator \mathcal{F}
    must be in the form [[x1,xs1],...,[xn,xsn]]
 z0 : input point
  p : order of cyclicity, must be at least two
z0s : generic point in F(z0)
*/
   z0s:makelist(Y[i],i,length(z0)), 
   domF:unique(xl(opF)),
   eqs:[],
   for z1 in delete(z0,domF) do(
      eqs:append(eqs,[Mtilde(opF,z0,z1,p)<=(z0-z1).z0s])
   ),
   eqs /* output */
)$
\end{lstlisting}

The following procedure was extracted from~\cite{cite:ext}.
It computes the extreme points of a polyhedron defined by a list of inequalities.
\begin{lstlisting}
ext(apr):=block([var,fs,cs,ap,s,S,m],
   load(simplex),
   var:sort(listofvars(apr)),
   s:apply("+",var),
   fs:append([1,s,-s],var,-var),
   ap(k):=subst(apr[k]=(lhs(apr[k])=rhs(apr[k])),apr),
   cs:makelist(ap(k),k,1,length(apr)),
   S:[],
   for f in fs do
      for c in cs do (
         m:minimize_lp(f,c),
         if listp(m) then
            S:cons(subst(m[2],var),S)
      ),
   listify(setify(S))
)$
\end{lstlisting}

The procedure \texttt{operatorTn} computes the finite operator $\mathcal{F}$ as in Proposition~\ref{pro:proptnp1}, item {\it 3}. The procedure takes as input the operator $T_0=\{(x_i,x_i^*)\}_{i=1}^n$ and proceeds in the following way: $\mathcal{F}$ starts as $T_0$ and, for each $x_k$ in $\{x_1,\ldots,x_n\}$,
\begin{itemize}
 \item using \texttt{polareqs}, it computes the equations that define $T_k(x_k)$;
 \item using \texttt{ext}, it computes $E_k$, the extreme points of $T_k(x_k)$;
 \item then it deletes $(x_k,x_k^*)$ from $\mathcal{F}$;
 \item and finally, it adds $\{x_k\}\times E_k$ to $\mathcal{F}$.
\end{itemize}

\begin{lstlisting}
operatorTn(opT0,p):=block([opF,domT0,xk,Tkxkeqs,Ek,ptk,zv],
/*
opT0 : list of points of initial points {(x_i,x_i^*)},
        must be in the form [[x1,xs1],...,[xn,xsn]]
*/
   opF:opT0,
   domT0:xl(opT0),
   for xk in domT0 do(
      Tkxkeqs:polareqs(opF,xk,p), /* equations of T_k(xk) */
      Ek:unique(ratsimp(ext(Tkxkeqs))), /* vertices E_k */
      for ptk in opF do(
         if xk=xc(ptk) then(
            opF:delete(ptk,opF)
        )
      ), /* deletion of (xk,xks) from F*/
      for zv in Ek do(
         opF:append(opF,[flatten([xk,zv])])  
      ) /* addition of xk \times Ek to F*/
   ),	
   opF /* output */
)$ 
\end{lstlisting}

\section{Examples}
In this section we present many different explicit examples of maximal $p$-cyclically monotone operators. First, a little outline of our procedure is given. For each example we present the \emph{starting points} $(x_1,x_1^*),\ldots,(x_n,x_n^*)$ and, using the algorithm given in the Section~\ref{sec:algorithm} and the implementation of $\widetilde{M}(z_0,z_1)$ given in Section~\ref{sec:tools}, we present the \emph{second to last operator} $T_n$, the operator $\mathcal{F}$ and the equations that any $(z_0,z_0^*)\in T_{n+1}$, $T_{n+1}$ being the \emph{last step operator}, must satisfy. 
Following that we try to obtain the domain and correspondence rule of $T_{n+1}$. In most cases, we make use of Corollary~\ref{cor:mtilde} to obtain a candidate for $\dom(T_{n+1})$. Finally, we prove that $T_{n+1}$ is $p$-cyclically monotone, which implies that it is maximal $p$-cyclically monotone.

\subsection{Bauschke and Wang original example}\label{sec:BW}
Consider the starting points:
\begin{align*}
x_1&=(1, 0),&x_1^*&=(0,1), \\
x_2&=(0, 1),&x_2^*&=(-1, 0),\\
x_3&=(-1,0),&x_3^*&=(-1,-2),\\
x_4&=(0,-1),&x_4^*&=(0,-1).
\end{align*}
After the first four steps of our algorithm, we obtain the following images of $x_1,\ldots,x_4$:
\begin{align*}
T_4(x_1)&=\{(u,v)\::\: u+1\geq |v|,\,u\geq 0\}\\
        &=\co\{(0,-1),(0,1)\}+\cone\co\{(1,-1),(1,1)\},\\
T_4(x_2)&=\{(u,v)\::\: v\geq |u+1|\}\\
        &=(-1,0)+\cone\co\{(1,1),(-1,1)\},\\
T_4(x_3)&=\{(u,v)\::\:u\leq -|v+1|,\,u\leq -1\}\\
        &=\co\{(-1,0),(-1,-2)\}+\cone\co\{(-1,1),(-1,-1)\},\\
T_4(x_4)&=\{(u,v)\::\: v\leq -|u|-1\}\\
        &=(0,-1)+\cone\co\{(-1,-1),(1,-1)\}.
\end{align*}
Therefore $\mathcal{F}=\{(y_i,y_i^*)\}_{i=1}^6$ is given by
\begin{align*}
y_1&=(1, 0)&y_1^*&=(0,-1),&y_2&=(1, 0)&y_2^*&=(0,1), \\
y_3&=(0, 1)&y_3^*&=(-1, 0),&y_4&=(-1,0)&y_4^*&=(-1,0),\\
y_5&=(-1,0)&y_5^*&=(-1,-2),&y_6&=(0,-1)&y_6^*&=(0,-1),
\end{align*}
that is 
\begin{align*}
\dom(\mathcal{F})&=\{(1,0),(0,1),(-1,0),(0,-1)\},\\\ran(\mathcal{F})&=\{(0,1),(0,-1),(-1,0),(-1,-2)\}. 
\end{align*}
We now consider $z_0=(x,y)$ and $z_0^*=(u,v)$, and follow the steps of the proof of Proposition~\ref{pro:opfinite}, so we obtain the set of inequalities:
\begin{align*}
 \max\{1-x,-x-2y-1,-y,y\}&\leq (x-1)u+yv,\\
 \max\{-x,-x-2y-2,-y-1,y-1\}&\leq xu+(y-1)v,\\
 \max\{-x-2,-x-2y,-y-1,y-1\}&\leq (x+1)u+yv,\\
 \max\{-x-1,-x-2y-1,-y,y-2\}&\leq xu+(y+1)v.
\end{align*}
These inequalities were obtained and handled by Bauschke and Wang in~\cite{Bauschke08}. They were able to prove that the last step operator $T_{5}$ is defined on $\{(x,y)\::\:|x|+|y|=1\}$ and
\begin{align*}
T_5(x_i)&=T_4(x_i),\quad\forall i=1,2,3,4,\\
T_5(1-t,t)&=(-t,1-t)+\cone\{(1,1)\},\\
T_5(-t,1-t)&=(-1,0)+\cone\{(-1,1)\},\\
T_5(-1+t,-t)&=(t-1,t-2)+\cone\{(-1,-1)\},\\
T_5(t,t-1)&=(0,-1)+\cone\{(1,-1)\}.
\end{align*}
See Figure~\ref{fig:BW} for a partial graphical representation of the domain and range of $T_5$.

\begin{figure}
\centering
\begin{tikzpicture}[scale=1.5]
\node at (0,2.5) {$\dom(T_5)$};
\draw[very thick] (1,0) -- (0,1) -- (-1,0) -- (0,-1) -- cycle;

\draw[semitransparent,->](-1.5,0)--(2,0)node[below right]{$x$};
\draw[semitransparent,->](0,-1.5)--(0,2)node[left]{$y$};

\fill
(1,0)circle(1.5pt)node[below right]{\small $x_1$};
\fill
(-1,0)circle(1.5pt)node[above left]{\small $x_3$};
\fill
(0,1)circle(1.5pt)node[right]{\small $x_2$};
\fill
(0,-1)circle(1.5pt)node[left]{\small $x_4$};

\end{tikzpicture}
\qquad%
\begin{tikzpicture}
\node at (0,2.6) {$\ran(T_5)$}; 
\fill[gray!30] 
(1,2) --  (0,1) -- (0,-1) -- (2,-3) -- (2,2) -- cycle; 
\fill[gray!30] 
(1,2) --  (-1,0) --  (-3,2) -- cycle;
\fill[gray!30] 
(-3,2) -- (-1,0) -- (-1,-2) -- (-2,-3) -- (-3,-3) -- cycle;
\fill[gray!30] 
(-2,-3) -- (0,-1) -- (2,-3) -- cycle;

\draw[very thick] (-3,2) -- (-1,0) -- (1,2);
\draw[very thick] (-2,-3) -- (0,-1) -- (2,-3);
\draw[very thick] (-1,0) -- (-1,-2);
\draw[very thick] (0,1) -- (0,-1);

\draw[dashed,semitransparent,->](-3,0)--(2.1,0)node[below right]{$u$};
\draw[dashed,semitransparent,->](0,-3)--(0,2.1)node[left]{$v$};

\node[above] 
at (1,0) {\scriptsize $T_5(x_1)$};
\node 
at (-1,1.25) {\scriptsize $T_5(x_2)$};
\node 
at (-2,-1) {\scriptsize $T_5(x_3)$};
\node 
at (0,-2) {\scriptsize $T_5(x_4)$};

\fill 
(0,1) circle(2pt) node[right]
{\footnotesize $x_1^*$};
\fill 
(-1,-2) circle(2pt) node[below]
{\footnotesize $x_3^*$};
\fill 
(-1,0) circle(2pt) node[above]
{\footnotesize $x_2^*$};
\fill 
(0,-1) circle(2pt) node[right] 
{\footnotesize $x_4^*$};

\end{tikzpicture}
\caption{Example from Bauschke and Wang~\cite{Bauschke08}}\label{fig:BW}
\end{figure}
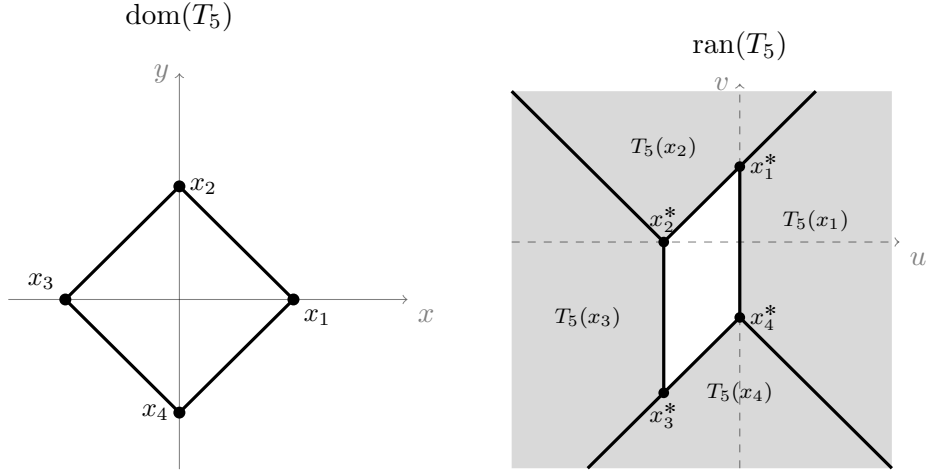

\subsection{A maximal 2-cyclically monotone operator in $\R^3$}
Consider the starting points:
\begin{align*}
x_1&=(-1, -1, -1),&x_1^*&=(-8,-8,16),\\
x_2&=(1, 0, 0),&x_2^*&=(8,12,0), \\
x_3&=(0, 1, 0),&x_3^*&=(-12, 8, 0),\\
x_4&=(0, 0, 1),&x_4^*&=(0, 0, 16).
\end{align*}
Note that $C=\co\{x_1,x_2,x_3,x_4\}$ forms a tetrahedron in $\R^3$. After the first four steps, we obtain the following images of $x_1,\ldots,x_4$, 
\begin{align*}
T_4(x_1)&=\{(u,v,w)\::\:2u+v+w\leq 0,\,u+2v+w\leq 4,\,u+v+2w\leq 20\}\\
        &=(-6,-2,14)+\cone\co\{(1,1,-3),(1,-3,1),(-3,1,1)\},\\
T_4(x_2)&=\{(u,v,w)\::\:2u+v+w\geq 4,\,v-u\leq 4,\,w-u\leq 16\}\\
        &=(-4,0,12)+\cone\co\{(1,1,1),(1,-3,1),(1,1,-3)\},\\
T_4(x_3)&=\{(u,v,w)\::\:u+2v+w\geq 4,\,v-u\geq 4,\,w-v\leq 12\}\\
        &=(-5,-1,11)+\cone\co\{(1,1,1),(-3,1,1),(1,1,-3)\},\\
T_4(x_4)&=\{(u,v,w)\::\:u+v+2w\geq 20,\,w-u\geq 16,\,w-v\geq 16\}\\
        &=(-3,-3,13)+\cone\co\{(1,1,1),(1,-3,1),(-3,1,1)\}.
\end{align*}
Therefore $\mathcal{F}=\{(y_i,y_i^*)\}_{i=1}^4$ is given by
\begin{align*}
y_1&=(1,0,0),&y_1^*&=(-4,0,12),&y_2&=(0, 1,0),&y_2^*&=(-5,-1,11),\\
y_3&=(-1,-1,-1),&y_3^*&=(-6,-2,14),&y_4&=(0,0,1),&y_4^*&=(-3,-3,13),
\end{align*}
that is 
\begin{align*}
\dom(\mathcal{F})&=\{(1,0,0),(0,1,0),(0,0,1),(-1,-1,-1)\},\\\ran(\mathcal{F})&=\{(-6,-2,14),(-4,0,12),(-5,-1,11),(-3,-3,13)\}. 
\end{align*}
We now consider $z_0=(x,y,z)$ and $z_0^*=(u,v,w)$, and follow the steps of the proof of Proposition~\ref{pro:opfinite}, so we obtain the set of inequalities:
\begin{equation}\label{eq:3d1234}
 \begin{aligned}
 \max\left\{ \begin{array}{c}13z-3y-3x+7\\14z-2y-6x+6\\11z-y-5x+5\\12 z-4x+4\end{array}\right\}&\leq (x+1)u+(y+1)v+(z+1)w,\\
 \max\left\{ \begin{array}{c}11z-y-5x+5\\12 z-4x+4\\13z-3y-3x+3\\14z-2y-6x+2\end{array} \right\}&\leq (x-1)u+yv+zw,\\
 \max\left\{ \begin{array}{c}14z-2y-6x+2\\11z-y-5x+1\\12 z-4x\\13z-3y-3x-1\end{array} \right\}&\leq xu+(y-1)v+zw,\\
 \max\left\{ \begin{array}{c}12 z-4x-12\\13z-3y-3x-13\\14z-2y-6x-14\\11z-y-5x-15\end{array} \right\}&\leq xu+yv+(z-1)w,
\end{aligned}
\end{equation}
together with the domain conditions
\begin{equation}\label{eq:3ddom}
x+y+z\leq 1,\quad x+y-3z\leq 1,\quad x-3y+z\leq 1,\quad -3x+y+z\leq 1.  
\end{equation}
Using~\eqref{eq:3ddom}, it is straightforward to verify that the terms inside the maximums taken on the left hand side of inequalities~\eqref{eq:3d1234} are in decreasing order, from top to bottom. Therefore
\begin{align}
\widetilde{M}(z_0,x_1):=13z-3y-3x+7&\leq (x+1)u+(y+1)v+(z+1)w,\label{eq:3dran1}\\
\widetilde{M}(z_0,x_2):=11z-y-5x+5&\leq (x-1)u+yv+zw,\label{eq:3dran2}\\
\widetilde{M}(z_0,x_3):=14z-2y-6x+2&\leq xu+(y-1)v+zw,\label{eq:3dran3}\\
\widetilde{M}(z_0,x_4):=12 z-4x-12&\leq xu+yv+(z-1)w.\label{eq:3dran4}
\end{align}

To determine the domain of $T_5$, recall that any $z_0\in C$ can be uniquely written as
\[
z_0=\sum_{i=1}^4\alpha_i x_i= (\alpha_2-\alpha_1,\alpha_3-\alpha_1,\alpha_4-\alpha_1),
\]
with $\alpha_i\geq 0$ and $\ds\sum_{i=1}^4\alpha_i=1$. By Corollary~\ref{cor:mtilde}, $z_0\in\dom(T_5)$ if, and only if, 
\[                                                                                                                                                                     
\sum_{i=1}^4\alpha_i \widetilde{M}(z_0,x_i)\leq 0.                                                                                                                                                                    \]
This inequality, after replacing in the formulas for $\widetilde{M}$ and considering $\alpha_1=1-\alpha_2-\alpha_3-\alpha_4$, takes the form
\[
\alpha_2(1-\alpha_2-\alpha_3-\alpha_4)+\alpha_3\alpha_4\leq 0,
\]
that is,
\[
\alpha_1\alpha_2+\alpha_3\alpha_4\leq 0.
\]
Since $\alpha_i\geq 0$, for all $i$, the last inequality holds if. and only if, $\alpha_1\alpha_2=0$ and $\alpha_3\alpha_4=0$. Therefore $z_0\in \dom(T_5)$ if, and only if,
\begin{align*}
\alpha_1=0\wedge \alpha_4=0 &\quad\iff\quad z_0\in  [x_2,x_3]=[y_1,y_2],\\
\alpha_2=0\wedge \alpha_4=0 &\quad\iff\quad z_0\in  [x_3,x_1]=[y_2,y_3],\\
\alpha_2=0\wedge \alpha_3=0 &\quad\iff\quad z_0\in  [x_1,x_4]=[y_3,y_4],\\
\alpha_1=0\wedge \alpha_3=0 &\quad\iff\quad z_0\in  [x_4,x_2]=[y_4,y_1].
\end{align*}
Note that, although $C=\co\{x_i\}$ is a tetrahedron, the domain of $T_5$ are the segments
\[
\dom(T_5)=[y_1,y_2]\cup[y_2,y_3]\cup[y_3,y_4]\cup [y_4,x_1].
\]
See Figure~\ref{fig:3dmax}.
\tdplotsetmaincoords{55}{100}
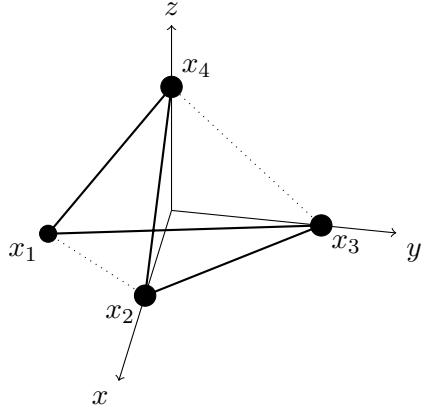
\begin{figure}
\centering
\begin{tikzpicture}[scale=2,tdplot_main_coords]
\coordinate (O) at (0,0,0);
\draw[->] (0,0,0) -- (2,0,0) node[anchor=north east]{$x$};
\draw[->] (0,0,0) -- (0,1.5,0) node[anchor=north west]{$y$};
\draw[->] (0,0,0) -- (0,0,1.5) node[anchor=south]{$z$};

\draw[fill] (0,0,1) circle (2pt) node[anchor=south west]{$x_4$};
\draw[fill] (0,1,0) circle (2pt) node[anchor=north west]{$x_3$} ;
\draw[fill] (1,0,0) circle (2pt) node[anchor=north east]{$x_2$};
\draw[fill] (-1,-1,-1) circle (1.6pt) node[anchor=north east]{$x_1$};
\draw[thick] (1,0,0) -- (0,1,0) -- (-1,-1,-1) -- (0,0,1) -- cycle;
\draw[dotted] (1,0,0) -- (-1,-1,-1);
\draw[dotted] (0,1,0) -- (0,0,1);
\end{tikzpicture}
\caption{Domain of a maximal 2-cyclically monotone operator in $\R^3$}\label{fig:3dmax}
\end{figure}

The operator $T_5$ coincides with $T_4$ at $y_1,\ldots,y_4$. We now deduce the formula for $T_5$ at the relative interior of each of the segments of the domain.
\begin{itemize}
\item $z_0\in ]y_1,y_2[\iff z_0=(t,1-t,0),\,t\in]0,1[$. Replacing this in equations~\eqref{eq:3dran1} through~\eqref{eq:3dran4}, we obtain:
 \begin{align*}
 \eqref{eq:3dran1}&\qquad\text{implies}\qquad u+2v+w+t(u-v)\geq 4,\\
 \eqref{eq:3dran2}&\qquad\text{implies}\qquad (t-1)(4+u-v)\geq 0,\\
 \eqref{eq:3dran3}&\qquad\text{implies}\qquad t(4+u-v)\geq 0,\\
 \eqref{eq:3dran4}&\qquad\text{implies}\qquad 12+v+t(4+u-v)\geq w.
 \end{align*}
 Therefore $v-u=4$ so $12+v\geq w$, $u+2v+w\geq 4+4t$ and
 \begin{align*}
 T_5(t,1-t,0)&=\{(u,v,w)\::\:v-u=4,\,v-w\geq -12,\,u+2v+w\geq 4+4t\}\\
 &=(t-5,t-1,t+11)+\cone\co\{(1,1,-3),(1,1,1)\}.
  \end{align*}
\item $z_0\in ]y_2,y_3[\iff z_0=(-t,1-2t,-t),\,t\in]0,1[$. Replacing this in equations~\eqref{eq:3dran1} through~\eqref{eq:3dran4}, we obtain:
 \begin{align*}
 \eqref{eq:3dran1}&\qquad\text{implies}\qquad (t-1)(u+2v+w-4)\leq 0,\\
 \eqref{eq:3dran2}&\qquad\text{implies}\qquad 4+u+t(u+2v+w-4)\leq v,\\
 \eqref{eq:3dran3}&\qquad\text{implies}\qquad t(u+2v+w-4)\leq 0,\\
 \eqref{eq:3dran4}&\qquad\text{implies}\qquad w+t(u+2v+w-8)\leq 12+v.
 \end{align*}
 Therefore $u+2v+w=4$ so $4+u\leq v$, $w-v\leq 12+4t$ and
 \begin{align*}
 T_5(-t,1-2t,-t)&=\{(u,v,w)\::\:u+2v+w=4,\,v-u\geq 4,\,v-w\geq -12-4t\}\\
                &=(-5-t,-1-t,11+3t)+\cone\co\{(-3,1,1),(1,1,-3)\}.
  \end{align*}

 \item $z_0\in ]y_3,y_4[\iff z_0=(-t,-t,1-2t),\,t\in]0,1[$. Replacing this in equations~\eqref{eq:3dran1} through~\eqref{eq:3dran4}, we obtain:
 \begin{align*}
 \eqref{eq:3dran1}&\qquad\text{implies}\qquad (t-1)(u+v+2w-20)\leq 0,\\
 \eqref{eq:3dran2}&\qquad\text{implies}\qquad 16+u+t(u+v+2w-16)\leq w,\\
 \eqref{eq:3dran3}&\qquad\text{implies}\qquad 16+v+t(u+v+2w-20)\leq w,\\
 \eqref{eq:3dran4}&\qquad\text{implies}\qquad t(u+v+2w-20)\leq 0.
 \end{align*}
 Therefore $u+v+2w=20$ so $16+u+4t\leq w$, $16+v\leq w$ and
 \begin{align*}
 T_5(-t,-t,1-2t)&=\{(u,v,w)\::\:u+v+2w=20,\,w-u\geq 16+4t,\,w-v\geq 16\}\\
                &=(-3-3t,t-3,13+t)+\cone\co\{(-3,1,1),(1,-3,1)\}.
  \end{align*}
  \item $z_0\in ]y_4,y_1[\iff z_0=(t,0,1-t),\,t\in]0,1[$. Replacing this in equations~\eqref{eq:3dran1} through~\eqref{eq:3dran4}, we obtain:
 \begin{align*}
 \eqref{eq:3dran1}&\qquad\text{implies}\qquad 20+t(w-u-16)\leq u+v+2w,\\
 \eqref{eq:3dran2}&\qquad\text{implies}\qquad (t-1)(16+u-w)\geq 0,\\
 \eqref{eq:3dran3}&\qquad\text{implies}\qquad 16+v+t(w-u-20)\leq w,\\
 \eqref{eq:3dran4}&\qquad\text{implies}\qquad t(16+u-w)\geq 0.
 \end{align*}
 Therefore $w-u=16$ so $16+v-4t\leq w$, $u+v+2w\geq 20$ and
 \begin{align*}
 T_5(t,0,1-t)&=\{(u,v,w)\::\:w-u=16,\,w-v\geq 16-4t,\,u+v+2w\geq 20\}\\
                &=(-3-t,3t-3,13-t)+\cone\co\{(1,-3,1),(1,1,1)\}.
  \end{align*}
\end{itemize}
Therefore, the full correspondence rule of $T_5$ is
\begin{align*}
T_5(-1,-1,-1)=(-6,-2,14)+\cone\co\{(1,1,-3),(1,-3,1),(-3,1,1)\},\\
T_5(1,0,0)=(-4,0,12)+\cone\co\{(1,1,1),(1,-3,1),(1,1,-3)\},\\
T_5(0,1,0)=(-5,-1,11)+\cone\co\{(1,1,1),(-3,1,1),(1,1,-3)\},\\
T_5(0,0,1)=(-3,-3,13)+\cone\co\{(1,1,1),(1,-3,1),(-3,1,1)\},
\end{align*}
and, for all $t\in]0,1[$,
\begin{align*}
T_5(t,1-t,0)=(t-5,t-1,t+11)+\cone\co\{(1,1,-3),(1,1,1)\},\\
T_5(-t,1-2t,-t)=(-5-t,-1-t,11+3t)+\cone\co\{(-3,1,1),(1,1,-3)\},\\
T_5(-t,-t,1-2t)=(-3-3t,t-3,13+t)+\cone\co\{(-3,1,1),(1,-3,1)\},\\
T_5(t,0,1-t)=(-3-t,3t-3,13-t)+\cone\co\{(1,-3,1),(1,1,1)\}.
\end{align*}

From the calculations above, we conclude that $T_5$ has the form
\[
T_5=\bigcup_{i=1}^4\co\{(y_i,y_i^*),(y_{i+1},y_{i+1}^*)\}+N_C,
\]
considering $(y_5,y_5^*)=(y_1,y_1^*)$. From Lemma~\ref{lem:orto}, $T_5$ is $2$-cyclically monotone and, since $T^5=T_4^{\mu_2}$, $T^5$ is maximal 2-cyclically monotone.

\subsection{A maximal $3$-cyclically monotone operator}
Consider the starting points:
\begin{align*}
x_1&=(1, 0),&x_1^*&=(1,1), \\
x_2&=(1, 1),&x_2^*&=(0,2), \\
x_3&=(0, 1),&x_3^*&=(-1, 1),\\
x_4&=(-1,0),&x_4^*&=(-1,-1),\\
x_5&=(0,-1),&x_5^*&=(1,-1).
\end{align*}

After the first five steps of our algorithm, we obtain the following images of $x_1,\ldots,x_4$,
\begin{align*}
T_5(x_1)&=\{(u,v)\::\: u \geq  v,\, u + v \geq  0,\, v \leq 2\}\\
        &=\co\{(0,0),(2,2)\}+\cone\co\{(1,-1),(1,0)\},\\
T_5(x_2)&=\{(u,v)\::\: u\geq 0,\,v\geq 2\}\\
        &=(0,2)+\cone\co\{(1,0),(0,1)\},\\
T_5(x_3)&=\{(u,v)\::\: 2+u\leq v,\, u+v\geq 0,\, u\leq 0\}\\
        &=\co\{(0,2),(-1,1)\}+\cone\co\{(0,1),(-1,1)\},\\
T_5(x_4)&=\{(u,v)\::\: u \leq -1,\, u + v \leq 0,\, u \leq v\}\\
        &=\co\{(-1,1),(-1,-1)\}+\cone\co\{(-1,1),(-1,-1)\},\\
T_5(x_5)&=\{(u,v)\::\:v \leq 0,\, v \leq u,\, u + v \leq 0\}\\
        &=(0,0)+\cone\co\{(-1,-1),(1,-1)\}.
\end{align*}
Therefore $\mathcal{F}=\{(y_i,y_i^*)\}_{i=1}^8$ is given by
\begin{align*}
y_1&=(1, 0),&y_1^*&=(0,0),&y_2&=(1, 0),&y_2^*&=(2,2), \\
y_3&=(1,1),&y_3^*&=(0,2),&y_4&=(0,1),&y_4^*&=(0,2),\\
y_5&=(0,1),&y_5^*&=(-1,1),&y_6&=(-1,0),&y_6^*&=(-1,1),\\
y_7&=(-1,0),&y_7^*&=(-1,-1),&y_8&=(0,-1),&y_8^*&=(0,0),
\end{align*}
that is 
\begin{align*}
\dom(\mathcal{F})&=\{(1,0),(1,1),(0,1),(-1,0),(0,-1)\},\\
\ran(\mathcal{F})&=\{(0,0),(2,2),(0,2),(-1,1),(-1,-1)\}.
\end{align*}
We now consider $z_0=(x,y)$ and $z_0^*=(u,v)$, and follow the steps of the proof of Proposition~\ref{pro:opfinite}, so we obtain the set of inequalities:
\begin{equation}\label{eq:mx3m1234}
 \begin{aligned}
  \max\left\{
 \begin{array}{c}
0, -y - x - 1, 2y,\\ y - x + 1, 2y +2x - 2
\end{array}
 \right\}&\leq (x-1)u+yv,\\
  \max\left\{
 \begin{array}{c}
 -2, -y - x - 1, 2y - 2,\\ y - x - 1, 2y + 2x - 4
 \end{array}
 \right\}&\leq (x-1)u+(y-1)v,\\
 \max\left\{
 \begin{array}{c}
 0, -y - x - 1, 2y - 2,\\ y - x - 1, 2y + 2x - 4
 \end{array}
 \right\}&\leq xu+(y-1)v,\\
  \max\left\{
 \begin{array}{c}
0, -y - x - 1, 2y - 2,\\ y - x - 1, 2y + 2x - 2
 \end{array}
 \right\}&\leq (x+1)u+yv,\\
 \max\left\{
 \begin{array}{c}
0, -y - x - 1, 2y,\\ y - x - 1, 2y + 2x -2
 \end{array}
 \right\}&\leq xu+(y+1)v.
 \end{aligned}
\end{equation}
In addition, we have the domain conditions:
\[
x-1\leq y\leq x+1,\qquad -x-1\leq y \leq 1,\qquad x\leq 1.
\]
These conditions allow us to further simplify the equations in~\eqref{eq:mx3m1234}:
\begin{align}
 \widetilde{M}(z_0,x_1):=\max\{2y,y-x+1\}&\leq (x-1)u+yv,\label{eq:mx3m1}\\
 \widetilde{M}(z_0,x_2):=
 \max\left\{
 \begin{array}{c}
 -x - y - 1\\2y - 2\\ y - x - 1
 \end{array}
 \right\}&\leq (x-1)u+(y-1)v,\label{eq:mx3m2}\\
  \widetilde{M}(z_0,x_3):=0&\leq xu+(y-1)v,\label{eq:mx3m3}\\
 \widetilde{M}(z_0,x_4):=\max\{0,2x+2y-2\}&\leq (x+1)u+yv,\label{eq:mx3m4}\\
 \widetilde{M}(z_0,x_5):=\max\{0,2y\} &\leq xu+(y+1)v.\label{eq:mx3m5}
\end{align}

To determine the domain of $T_6$, recall that any $z_0\in C$ can be written as
\[
z_0=\sum_{i=1}^5\alpha_i x_i= (\alpha_1+\alpha_2-\alpha_4,\alpha_2+\alpha_3-\alpha_5),
\]
with $\alpha_i\geq 0$ and $\ds\sum_{i=1}^4\alpha_i=1$. 
\begin{itemize}
 \item Let $z_0\in\co\{x_1,x_2,x_3\}$ that is, $\alpha_4=\alpha_5=0$, $\alpha_1+\alpha_2+\alpha_3=1$, $\alpha_1,\alpha_2,\alpha_3\geq 0$ and $z_0=(\alpha_1+\alpha_2,\alpha_2+\alpha_3)=(\alpha_1+\alpha_2,1-\alpha_1)$. Thus
 \begin{align*}
  \sum_{i=1}^5\alpha_i \widetilde{M}(z_0,x_i)&=(1-\alpha_1-\alpha_2)0+\alpha_1\max\{2-2\alpha_1,2-2\alpha_1-\alpha_2\}\\
&\phantom{=}+\alpha_2\max\{-\alpha_2-2,-2\alpha_1,-2\alpha_1-\alpha_2\}\\
&=\alpha_1(2-2\alpha_1)+\alpha_2(-2\alpha_1)\\
&=2\alpha_1(1-\alpha_1-\alpha_2). 
 \end{align*}
We conclude that $\ds\sum_{i=1}^5\alpha_i \widetilde{M}(z_0,x_i)>0$ if, and only if, $\alpha_1>0$ and $1-\alpha_1-\alpha_2>0$. Since  $z_0=(x,y)=(\alpha_1+\alpha_2,1-\alpha_1)$, we obtain $x=\alpha_1+\alpha_2<1$ and $y=1-\alpha_1<1$. Therefore, all $z_0\in\co\{x_1,x_2,x_3\}$ such that $x<1$ and $y<1$ do not belong to $\dom(T_6)$.
 \item Let $z_0\in\co\{x_1,x_3,x_4\}$ that is, $\alpha_2=\alpha_5=0$, $\alpha_1+\alpha_3+\alpha_4=1$, $\alpha_1,\alpha_3,\alpha_4\geq 0$ and $z_0=(\alpha_1-\alpha_4,\alpha_3)=(\alpha_1-\alpha_4,1-\alpha_1-\alpha_4)$. Thus
 \begin{align*}
  \sum_{i=1}^5\alpha_i \widetilde{M}(z_0,x_i)&=\alpha_1\max\{2(1-\alpha_1-\alpha_4),2(1-\alpha_1)\}\\
  &\phantom{=}+(1-\alpha_1-\alpha_4)0+\alpha_4\max\{0,-4\alpha_4\}\\
  &=2\alpha_1(1-\alpha_1)
 \end{align*}
We conclude that $\ds\sum_{i=1}^5\alpha_i \widetilde{M}(z_0,x_i)>0$ if, and only if, $0<\alpha_1<1$. Since  $z_0=(x,y)=(\alpha_1-\alpha_4,1-\alpha_1-\alpha_4)$, we obtain $0<2\alpha_1=1+x-y<2$. Therefore, all $z_0\in\co\{x_1,x_3,x_4\}$ such that $x-1<y<x+1$ do not belong to $\dom(T_6)$.
\item Let $z_0\in\co\{x_1,x_4,x_5\}$ that is, $\alpha_2=\alpha_3=0$, $\alpha_1+\alpha_4+\alpha_5=1$, $\alpha_1,\alpha_4,\alpha_5\geq 0$ and $z_0=(\alpha_1-\alpha_4,-\alpha_5)=(\alpha_1-\alpha_4,\alpha_1+\alpha_4-1)$. Thus
 \begin{align*}
  \sum_{i=1}^5\alpha_i \widetilde{M}(z_0,x_i)&=\alpha_1\max\{2\alpha_4-2(1-\alpha_1),2\alpha_4\}\\
  &\phantom{=}+\alpha_4\max\{0,-4(1-\alpha_1)\}\\
  &\phantom{=}+(1-\alpha_1-\alpha_4)\max\{0,-2(1-\alpha_1-\alpha_4)\}\\
  &=2\alpha_1\alpha_4.
 \end{align*}
We conclude that $\ds\sum_{i=1}^5\alpha_i \widetilde{M}(z_0,x_i)>0$ if, and only if, $\alpha_1>0$ and $\alpha_4>0$. Since $z_0=(x,y)=(\alpha_1-\alpha_4,\alpha_1+\alpha_4-1)$, we obtain $2\alpha_1=1+x+y>0$ and $2\alpha_4=1-x+y>0$. Therefore, all $z_0\in\co\{x_1,x_3,x_4\}$ such that $-y-1<x<y+1$ do not belong to $\dom(T_6)$.
\end{itemize}
In view of the above calculations, we obtain that
\[
\dom(T_6)\subset[x_1,x_2]\cup[x_2,x_3]\cup[x_3,x_4]\cup[x_4,x_5]\cup[x_5,x_1].
\]
We now are going to prove the converse inclusion and, at the same time, find the correspondence rule of $T_6$. This will be done in several parts.
\begin{itemize}
 \item Let $z_0\in]x_1,x_2[$, that is $z_0=(1,t)$, $t\in]0,1[$. Replacing this in equations~\eqref{eq:mx3m1} through~\eqref{eq:mx3m5}, we obtain:
 \begin{align*}
 \eqref{eq:mx3m1}&\qquad\text{implies}\qquad 2\leq v,\\
 \eqref{eq:mx3m2}&\qquad\text{implies}\qquad 2\geq v,\\
 \eqref{eq:mx3m3}&\qquad\text{implies}\qquad 0\leq u+(t-1)v,\\
 \eqref{eq:mx3m4}&\qquad\text{implies}\qquad 2t\leq 2u+tv,\\
 \eqref{eq:mx3m5}&\qquad\text{implies}\qquad 2t\leq u+(t+1)v.
 \end{align*}
That is $v=2$ and $u\geq 2(1-t)$. Therefore
 \begin{align*}
 T_6(1,t)&=\{(u,v)\::\:v=2,\,u\geq 2-2t\}\\
 &=(2-2t,2)+\cone\{(1,0)\}.
 \end{align*}
 \item Let $z_0\in]x_2,x_3[$, that is $z_0=(1-t,1)$, $t\in]0,1[$. Replacing this in equations~\eqref{eq:mx3m1} through~\eqref{eq:mx3m5}, we obtain:
 \begin{align*}
 \eqref{eq:mx3m1}&\qquad\text{implies}\qquad 2\leq -tu+v,\\
 \eqref{eq:mx3m2}&\qquad\text{implies}\qquad u\leq 0,\\
 \eqref{eq:mx3m3}&\qquad\text{implies}\qquad 0\leq u,\\
 \eqref{eq:mx3m4}&\qquad\text{implies}\qquad 2(1-t) \leq (2-t)u+v,\\
 \eqref{eq:mx3m5}&\qquad\text{implies}\qquad 2\leq (1-t)u+2v.
 \end{align*}
That is $u=0$ and $v\geq 2$. Therefore
 \begin{align*}
 T_6(1-t,1)&=\{(u,v)\::\:u=0,\,v\geq 2\}\\
 &=(0,2)+\cone\{(0,1)\}.
 \end{align*}
 \item Let $z_0\in]x_3,x_4[$, that is $z_0=(-t,1-t)$, $t\in]0,1[$. Replacing this in equations~\eqref{eq:mx3m1} through~\eqref{eq:mx3m5}, we obtain:
 \begin{align*}
 \eqref{eq:mx3m1}&\qquad\text{implies}\qquad 2\leq -(t+1)u+(1-t)v,\\
 \eqref{eq:mx3m2}&\qquad\text{implies}\qquad 0\leq -(t+1)u-tv,\\
 \eqref{eq:mx3m3}&\qquad\text{implies}\qquad 0\geq u+v,\\
 \eqref{eq:mx3m4}&\qquad\text{implies}\qquad 0\leq u+v,\\
 \eqref{eq:mx3m5}&\qquad\text{implies}\qquad 2(1-t) \leq -tu+(2-t)v.
 \end{align*}
That is $v=-u$ and $u\leq -1$. Therefore
 \begin{align*}
 T_6(-t,1-t)&=\{(u,v)\::\:v=-u,\,u\leq -1\}\\
 &=(-1,1)+\cone\{(-1,1)\}.
 \end{align*}
  \item Let $z_0\in]x_4,x_5[$, that is $z_0=(t-1,-t)$, $t\in]0,1[$. Replacing this in equations~\eqref{eq:mx3m1} through~\eqref{eq:mx3m5}, we obtain:
 \begin{align*}
 \eqref{eq:mx3m1}&\qquad\text{implies}\qquad 2-2t\leq (t-2)u-tv,\\
 \eqref{eq:mx3m2}&\qquad\text{implies}\qquad 0\leq (t-2)u-(t+1)v,\\
 \eqref{eq:mx3m3}&\qquad\text{implies}\qquad 0\leq (t-1)u-(t+1)v,\\
 \eqref{eq:mx3m4}&\qquad\text{implies}\qquad 0\leq u-v,\\
 \eqref{eq:mx3m5}&\qquad\text{implies}\qquad 0\leq -u+v.
 \end{align*}
  That is $u=v$ and $u\leq t-1$. Therefore
 \begin{align*}
 T_6(t-1,-t)&=\{(u,v)\::\:u=v,\,u\leq t-1\}\\
 &=(t-1,t-1)+\cone\{(-1,-1)\}.
 \end{align*} 
 \item Let $z_0\in]x_5,x_1[$, that is $z_0=(t,t-1)$, $t\in]0,1[$. Replacing this in equations~\eqref{eq:mx3m1} through~\eqref{eq:mx3m5}, we obtain:
 \begin{align*}
 \eqref{eq:mx3m1}&\qquad\text{implies}\qquad 0\leq -u-v,\\
 \eqref{eq:mx3m2}&\qquad\text{implies}\qquad -2t\leq (t-1)u+(t-2)v,\\
 \eqref{eq:mx3m3}&\qquad\text{implies}\qquad 0\leq tu+(t-2)v,\\
 \eqref{eq:mx3m4}&\qquad\text{implies}\qquad 0\leq (t+1)u+(t-1)v,\\
 \eqref{eq:mx3m5}&\qquad\text{implies}\qquad 0\leq u+v.
 \end{align*} 
 That is $v=-u$ and $u\geq 0$. Therefore
 \begin{align*}
 T_6(t,t-1)&=\{(u,v)\::\:v=-u,\,u\geq 0\}\\
 &=(0,0)+\cone\{(1,-1)\}.
 \end{align*} 
\end{itemize}
Therefore
\[
\dom(T_6)=[x_1,x_2]\cup[x_2,x_3]\cup[x_3,x_4]\cup[x_4,x_5]\cup[x_5,x_1],
\]
and
\begin{align*}
T_6(1,0)&=\co\{(0,0),(2,2)\}+\cone\co\{(1,-1),(1,0)\},\\
T_6(1,1)&=(0,2)+\cone\co\{(1,0),(0,1)\},\\
T_6(0,1)&=\co\{(0,2),(-1,1)\}+\cone\co\{(0,1),(-1,1)\},\\
T_6(-1,0)&=\co\{(-1,1),(-1,-1)\}+\cone\co\{(-1,1),(-1,-1)\},\\
T_6(0,-1)&=(0,0)+\cone\co\{(-1,-1),(1,-1)\},
\end{align*}
and, for every $t\in]0,1[$,
\begin{align*}
T_6(1,t)&=(2-2t,2)+\cone\{(1,0)\},\\
T_6(1-t,1)&=(0,2)+\cone\{(0,1)\},\\
T_6(-t,1-t)&=(-1,1)+\cone\{(-1,1)\},\\
T_6(t-1,-t)&=(t-1,t-1)+\cone\{(-1,-1)\},\\
T_6(t,t-1)&=(0,0)+\cone\{(1,-1)\}.
\end{align*}
See Figure~\ref{fig:mx3m} for a graphical representation of the domain and range of $T_6$.

Note that $T_6$ can be written as
\[
T_6=S+N_C
\]
where $S=\ds\bigcup_{i=1}^8\co\{(y_i,y_i^*),(y_{i+1},y_{i+1}^*)\}$, considering $(y_9,y_9^*)=(y_1,y_1^*)$.
From Lemma~\ref{lem:orto}, the operator $T_6$ is 3-cyclically monotone and, since $T_6=T_5^{\mu_3}$, $T_6$ is maximal 3-cyclically monotone.

On the other hand, also from Lemma~\ref{lem:orto} and Proposition~\ref{pro:copolar},
\[
(0,0)\in\mathcal{F}^{\mu_2}|_C=S^{\mu_2}|_C=(S+N_C)^{\mu_2}=T_6^{\mu_2}.
\]
However $(0,0)\notin T_6$. Therefore, $T_6$ is not maximal 2-cyclically monotone.
\begin{figure}\centering
\begin{tikzpicture}[scale=1.5]
\node at (0,2.5) {$\dom(T_6)$};

\draw[semitransparent,->](-1.5,0)--(2,0)node[below right]{$x$};
\draw[semitransparent,->](0,-1.5)--(0,2)node[left]{$y$};
\draw[very thick] (1,0) -- (1,1) -- (0,1) -- (-1,0) -- (0,-1) -- cycle;

\fill
(1,0)circle(2pt)node[below]{\small $x_1$};
\fill
(-1,0)circle(2pt)node[below]{\small $x_4$};
\fill
(0,1)circle(2pt)node[above right]{\small $x_3$};
\fill
(0,-1)circle(2pt)node[left]{\small $x_5$};
\fill
(1,1)circle(2pt)node[right]{\small $x_2$};

\end{tikzpicture}
\qquad
\begin{tikzpicture}[scale=0.83]
\node at (0,5) {$\ran(T_6)$}; 

\fill[gray!30] 
(-3,3)--(-1,1)--(0,2)--(0,4)--(-3,4)-- cycle;
\fill[gray!30] 
(2,-2)--(0,0)--(2,2)--(3,2)--(3,-2) -- cycle;
\fill[gray!30] 
(0,4)--(0,2)--(3,2)--(3,4) -- cycle;
\fill[gray!30] 
(-2,-2)--(0,0)--(2,-2)-- cycle;
\fill[gray!30] 
(-3,3)--(-1,1)--(-1,-1)--(-2,-2)--(-3,-2) -- cycle;

\draw[very thick] (-2,-2) -- (0,0) -- (2,-2);
\draw[very thick] (0,0) -- (2,2) -- (3,2);
\draw[very thick] (2,2) -- (0,2);
\draw[very thick] (-3,3) -- (-1,1)-- (0,2) -- (0,4);
\draw[very thick] (-1,1) -- (-1,-1);



\draw[dashed,semitransparent,->](-3,0)--(3.1,0)node[below right]{$u$};
\draw[dashed,semitransparent,->](0,-2)--(0,4.1)node[left]{$v$};

\node[above] 
at (2,0) {\scriptsize $T_6(x_1)$};
\node[right] 
at (1,3) {\scriptsize $T_6(x_2)$};
\node 
at (-1.5,3) {\scriptsize $T_6(x_3)$};
\node[above] 
at (-2,0) {\scriptsize $T_6(x_4)$};
\node 
at (0,-1.5) {\scriptsize $T_6(x_5)$};

\fill 
(1,1) circle(2.5pt) node[right]
{\footnotesize $x_1^*$};
\fill 
(0,2) circle(2.5pt) node[below]
{\footnotesize $x_2^*$};
\fill 
(-1,1) circle(2.5pt) node[above]
{\footnotesize $x_3^*$};
\fill 
(-1,-1) circle(2.5pt) node[left] 
{\footnotesize $x_4^*$};
\fill 
(1,-1) circle(2.5pt) node[right] 
{\footnotesize $x_5^*$};


\end{tikzpicture}
\caption{A maximal 3-cyclically monotone operator}\label{fig:mx3m}
\end{figure}
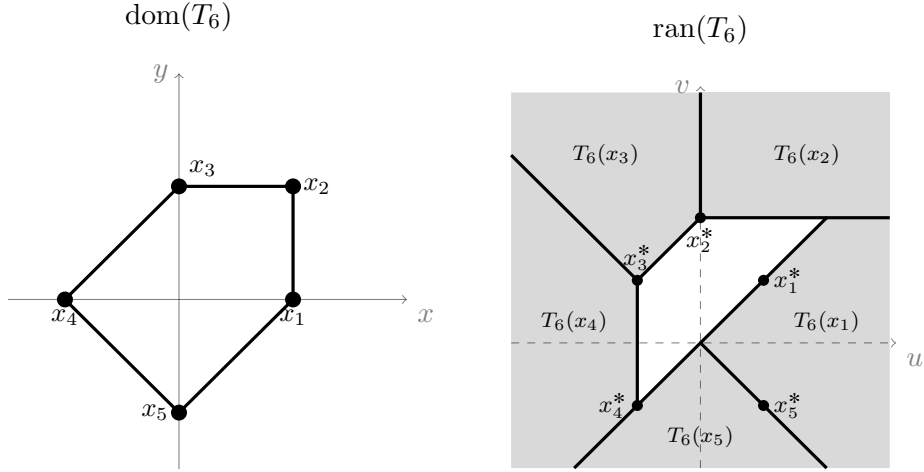

\subsection{A perturbation of Bauschke and Wang's example}
Consider the starting points:
\begin{align*}
x_1&=(1, 0),&x_1^*&=(0,1), \\
x_2&=(0, 1),&x_2^*&=(-1, 0),\\
x_3&=(-1,0),&x_3^*&=(-2,-2),\\
x_4&=(0,-1),&x_4^*&=(0,-1).
\end{align*}
These points are the same as in Section~\ref{sec:BW}, but considering $x_3^*=(-2,-2)$ instead of $(-1,-2)$.

After the first four steps of our algorithm, we obtain the following images of $x_1,\ldots,x_4$, 
\begin{align*}
T_4(x_1)&=\{(u,v)\::\: u\geq -1/2,\,-u-1\leq v\leq 1+u\}\\
        &=\co\{(-1/2,-1/2),(-1/2,1/2)\}+\cone\co\{(1,-1),(1,1)\},\\
T_4(x_2)&=\{(u,v)\::\:v\geq 0,\,-v-2\leq u\leq v-1 \}\\
        &=\co\{(-1,0),(-2,0)\}+\cone\co\{(1,1),(-1,1)\},\\
T_4(x_3)&=\{(u,v)\::\:u\leq -3/2,\,u\leq v\leq -u-2\}\\
        &=\co\{(-3/2,-1/2),(-3/2,-3/2)\}+\cone\co\{(-1,1),(-1,-1)\},\\
T_4(x_4)&=\{(u,v)\::\:v\leq -1,\,v\leq u\leq -v-1\}\\
        &=\co\{(-1,-1),(0,-1)\}+\cone\co\{(-1,-1),(1,-1)\}.
\end{align*}

Therefore $\mathcal{F}=\{(y_i,y_i^*)\}_{i=1}^8$ is given by
\begin{align*}
y_1&=(1, 0),&y_1^*&=(-1/2,-1/2),&y_2&=(1, 0),&y_2^*&=(-1/2,1/2), \\
y_3&=(0, 1),&y_3^*&=(-1, 0),&y_4&=(0,1),&y_4^*&=(-2,0),\\
y_5&=(-1,0),&y_5^*&=(-3/2,-1/2),&y_6&=(-1,0),&y_6^*&=(-3/2,-3/2),\\
y_7&=(0,-1),&y_7^*&=(-1,-1),&y_8&=(0,-1),&y_8^*&=(0,-1),
\end{align*}
that is 
\begin{align*}
\dom(\mathcal{F})&=\{(1,0),(0,1),(-1,0),(0,-1)\},\\
\ran(\mathcal{F})&=\left\{
\begin{array}{c}
(-1/2,-1/2),(-1/2,1/2),(-1,0),(-2,0),\\
(-3/2,-1/2),(-3/2,-3/2),(-1,-1),(0,-1)
\end{array}
\right\}. 
\end{align*}
We now consider $z_0=(x,y)$ and $z_0^*=(u,v)$, and follow the steps of the proof of Proposition~\ref{pro:opfinite}, so we obtain the set of inequalities:
\begin{equation}\label{eq:bwm1234}
 \begin{aligned}
  \max\left\{
 \begin{array}{c}
  1 - 2x, 1 - x, -x-y, -y,\\
  (y - x + 1)/2, -(x+y- 1)/2,\\
  -(3x+y + 1)/2, -(3x + 3y + 1)/2
\end{array}
 \right\}&\leq (x-1)u+yv,\\
  \max\left\{
 \begin{array}{c}
 -2x, -x, -y - 1, -y - x - 1,\\
 (y - x - 1)/2, -(x+y + 1)/2,\\
 -(3x+y - 1)/2, -(3x + 3y - 1)/2
 \end{array}
 \right\}&\leq xu+(y-1)v,\\
  \max\left\{
 \begin{array}{c}
  -x-2,-2x-2,-y-1,-y-x-1,\\
  (y-x-5)/2,-3(x+y+1)/2,\\
  -(x+y+5)/2,-(3x+y+3)/2
 \end{array}
 \right\}&\leq (x+1)u+yv,\\
 \max\left\{
 \begin{array}{c}
 -x - 2, -2x-2, -y - 1, -y - x - 1,\\
 (y - x - 1)/2, -3(x + y + 1)/2,\\
 -(x+y + 1)/2, -(3x+y + 3)/2
 \end{array}
 \right\}&\leq xu+(y+1)v.
 \end{aligned}
\end{equation}
In addition, we have the domain conditions:
\[
x-1\leq y\leq x+1,\qquad -x-1\leq y \leq 1-x,
\]
which allow us to obtain from~\eqref{eq:bwm1234} the simplified inequalities:
\begin{align}
 \widetilde{M}(z_0,x_1):=\max\{1-2x,1-x\}&\leq (x-1)u+yv,\label{eq:bwp1}\\
 \widetilde{M}(z_0,x_2):=\max\{(1-3x-3y)/2,(1-3x-y)/2\}&\leq xu+(y-1)v,\label{eq:bwp2}\\
 \widetilde{M}(z_0,x_3):=\max\{-1-y,-1-x-y\} &\leq (x+1)u+yv,\label{eq:bwp3}\\
 \widetilde{M}(z_0,x_4):=\max\{(-1-x-y)/2,(-1-x+y)/2\} &\leq xu+(y+1)v.\label{eq:bwp4}
\end{align}

To determine the domain of $T_5$, let $z_0\in C$ and write
\[
z_0=\sum_{i=1}^4\alpha_i x_i= (\alpha_1-\alpha_3,\alpha_2-\alpha_4),
\]
with $\alpha_i\geq 0$ and $\ds\sum_{i=1}^4\alpha_i=1$. By Corollary~\ref{cor:mtilde}, $z_0\in\dom(T_5)$ if, and only if, 
\[                                                                                                                                                                     
\sum_{i=1}^4\alpha_i \widetilde{M}(z_0,x_i)\leq 0.                                                                                                                                                                    \]
The determination of $\dom(T_5)$ will be done in several steps.
\begin{itemize}
 \item $(0,0)\in\dom(T_5)$. It is enough to replace $(x,y)=(0,0)$ on equations~\eqref{eq:bwp1} through~\eqref{eq:bwp4}, so we obtain
\[
1\leq -u,\quad -1\leq u,\quad 1/2\leq-v,\quad -1/2\leq v.
\]
 That is $(-1,-1/2)\in T_5(0,0)$ and $(0,0)\in\dom(T_5)$
\item $z_0\in]x_1,x_3[$ that is $z_0=\alpha_1(1,0)+\alpha_3(-1,0)=(\alpha_1-\alpha_3,0)$, with $\alpha_1,\alpha_3>0$ and $\alpha_1+\alpha_3=1$. Thus
 \begin{align*}
 \sum_{i=1}^4&\alpha_i \widetilde{M}(z_0,x_i)\\
 &=\alpha_1\max\{3-4\alpha_1,2-2\alpha_1\}+(1-\alpha_1)\max\{-1,-2\alpha_1\}\\
 &=(2\alpha_1-1)(1-\alpha_1)+\max\{0,1-2\alpha_1\}\\
 &=\begin{cases}
    \alpha_1(1-2\alpha_1),&\text{if }\alpha_1\leq 1/2,\\
    (2\alpha_1-1)(1-\alpha_1),&\text{if }\alpha_1>1/2.
   \end{cases}
 \end{align*}
Therefore, $\ds\sum_{i=1}^4\alpha_i \widetilde{M}(z_0,x_i)>0$ whenever $\alpha_1\neq 1/2$. This implies that $]x_1,x_3[\cap\dom(T_5)=\{(0,0)\}$.

\item Let $z_0\in\co\{x_1,x_2,x_4\}$, that is $\alpha_3=0$ and $\alpha_1+\alpha_2+\alpha_4=1$ with $\alpha_1,\alpha_2,\alpha_4\geq 0$. Therefore $z_0=(1-\alpha_2-\alpha_4,\alpha_2-\alpha_4)$ and
 \begin{align*}
 \sum_{i=1}^4&\alpha_i \widetilde{M}(z_0,x_i)\\
 &=\alpha_2\max\{\alpha_2 + 2 \alpha_4-1, 3 \alpha_4-1\}+\alpha_4\max\{\alpha_2-1,\alpha_4-1\}\\
 &\phantom{=}+(1-\alpha_2-\alpha_4)\max\{-1+2\alpha_2+2\alpha_4,\alpha_2+\alpha_4\}\\
 &=\alpha_2(2\alpha_4-1+\max\{\alpha_2,\alpha_4\})+\alpha_4(-1+\max\{\alpha_2,\alpha_4\})\\
 &\phantom{=}+(1-\alpha_2-\alpha_4)(\alpha_2+\alpha_4+\max\{-1+\alpha_2+\alpha_4,0\})\\
 &=-\alpha_2^2-\alpha_4^2+(\alpha_2+\alpha_4)\max\{\alpha_2,\alpha_4\},
 \end{align*}
 that is
 \[
  \sum_{i=1}^4\alpha_i \widetilde{M}(z_0,x_i)=
  \begin{cases}
  \alpha_4(\alpha_2-\alpha_4),&\text{if }\alpha_2\geq\alpha_4,\\
  \alpha_2(\alpha_4-\alpha_2),&\text{if }\alpha_2<\alpha_4.
  \end{cases}
 \]
Therefore $\ds\sum_{i=1}^4\alpha_i \widetilde{M}(z_0,x_i)\leq 0$ if, and only if, $\alpha_2=0$ and $\alpha_4>0$, or $\alpha_4=\alpha_2$, or $\alpha_4=0$ and $\alpha_2>0$.
\begin{itemize}
\item If $\alpha_2=0$ then $z_0\in[x_4,x_1]$. In this case, the coefficients $\alpha_i$ are unique, so this implies that $z_0\in\dom(T_5)$.
\item If $\alpha_4=0$ then $z_0\in[x_1,x_2]$. In this case, the coefficients $\alpha_i$ are also unique, so this implies that $z_0\in\dom(T_5)$.
\item If $\alpha_2=\alpha_4$ then $z_0$ is in the $x$-axis. We already proved that $z_0$ cannot be in $\dom(T_5)$ unless $z_0=(0,0)$.
\end{itemize}
In any other case, $z_0\notin\dom(T_5)$. This in particular implies that  the interior of the triangle $\co\{x_1,x_2,x_4\}$ and its vertical side, without the origin, does not intersect the domain of $T_5$.

\item Let $z_0\in\co\{x_2,x_3,x_4\}$, that is $\alpha_1=0$ and $\alpha_2+\alpha_3+\alpha_4=1$ with $\alpha_2,\alpha_3,\alpha_4\geq 0$. Therefore $z_0=(\alpha_2+\alpha_4-1,\alpha_2-\alpha_4)$ and
 \begin{align*}
 \sum_{i=1}^4&\alpha_i \widetilde{M}(z_0,x_i)\\
 &=\alpha_2\max\{2-3\alpha_2,2-2\alpha_2-\alpha_4\}+\alpha_4\max\{-\alpha_2,-\alpha_4\}\\
 &\phantom{=}+(1-\alpha_2-\alpha_4)\max\{-2\alpha_2,-1-\alpha_2+\alpha_4\}\\
 &=\alpha_2(2(1-\alpha_2)+\max\{-\alpha_2,-\alpha_4\})+\alpha_4\max\{-\alpha_2,-\alpha_4\}\\
 &\phantom{=}+(1-\alpha_2-\alpha_4)(-2\alpha_2+\max\{0,-1+\alpha_2+\alpha_4\})\\
 &=2\alpha_2\alpha_4+(\alpha_2+\alpha_4)\max\{-\alpha_2,-\alpha_4\},
 \end{align*}
 that is
 \[
  \sum_{i=1}^4\alpha_i \widetilde{M}(z_0,x_i)=
  \begin{cases}
  \alpha_4(\alpha_2-\alpha_4),&\text{if }\alpha_2\geq\alpha_4,\\
  \alpha_2(\alpha_4-\alpha_2),&\text{if }\alpha_2<\alpha_4.
  \end{cases}
 \]
Again $\ds\sum_{i=1}^4\alpha_i \widetilde{M}(z_0,x_i)\leq 0$ if, and only if, $\alpha_2=0$ and $\alpha_4>0$, or $\alpha_4=\alpha_2$, or $\alpha_4=0$ and $\alpha_2>0$. As before, this implies that the segments $[x_2,x_3]$ and $[x_3,x_4]$ belong to the domain of $T_5$ and no point in the interior of the triangle $\co\{x_2,x_3,x_4\}$ belongs to the domain of $T_5$.
\end{itemize}
In conclusion, 
\[
\dom(T_5)=\{(0,0)\}\cup [x_1,x_2]\cup [x_2,x_3] \cup [x_3,x_4] \cup [x_4,x_1].
\]
We now will deduce the formula of $T_5$.
\begin{itemize}
 \item $T_5(0,0)=(-1,-1/2)$, as we already verified.
\item $z_0\in ]x_1,x_2[\iff z_0=(1-t,t),\,t\in]0,1[$. Replacing this in equations~\eqref{eq:bwp1} through~\eqref{eq:bwp4}, we obtain:
 \begin{align*}
 \eqref{eq:bwp1}&\qquad\text{implies}\qquad 1\leq v-u,\\
 \eqref{eq:bwp2}&\qquad\text{implies}\qquad -1\leq u-v,\\
 \eqref{eq:bwp3}&\qquad\text{implies}\qquad -t-1\leq (2-t)u+tv,\\
 \eqref{eq:bwp4}&\qquad\text{implies}\qquad t-1\leq (1-t)u+(1+t)v.
 \end{align*}
 That is $v=1+u$, $u\geq -1$, $u\geq -t-1/2$. Therefore
 \begin{align*}
 T_5(1-t,t)&=\{(u,v)\::\:v=1+u,\,u\geq \max\{-1,-t-1/2\}\}\\
 &=\begin{cases}
    \{(u,v)\::\:v=1+u,\,u\geq -t-1/2\},&\text{if }t\leq 1/2\\
    \{(u,v)\::\:v=1+u,\,u\geq -1\},&\text{if }t>1/2
   \end{cases}\\
 &=\begin{cases}
    (-t-1/2,-t+1/2)+\cone\{(1,1)\},&\text{if }t\leq 1/2\\
    (-1,0)+\cone\{(1,1)\},&\text{if }t>1/2\\
   \end{cases}
 \end{align*}

 \item $z_0\in ]x_2,x_3[\iff z_0=(-t,1-t),\,t\in]0,1[$. Replacing this in equations~\eqref{eq:bwp1} through~\eqref{eq:bwp4}, we obtain:
 \begin{align*}
 \eqref{eq:bwp1}&\qquad\text{implies}\qquad 1+2t\leq -(1+t)u+(1-t)v,\\
 \eqref{eq:bwp2}&\qquad\text{implies}\qquad 2\leq -u-v,\\
 \eqref{eq:bwp3}&\qquad\text{implies}\qquad -2\leq u+v,\\
 \eqref{eq:bwp4}&\qquad\text{implies}\qquad 0\leq -tu+(2-t)v.
 \end{align*}
 That is $u+v=-2$, $v\geq -1/2$, $v\geq -t$. Therefore
 \begin{align*}
 T_5(-t,1-t)&=\{(u,v)\::\:u+v=-2,\,v\geq \max\{-1/2,-t\}\}\\
 &=\begin{cases}
    \{(u,v)\::\:u+v=-2,\,v\geq -t\},&\text{if }t\leq 1/2\\
    \{(u,v)\::\:u+v=-2,\,v\geq -1/2\},&\text{if }t>1/2
   \end{cases}\\
 &=\begin{cases}
    (t-2,-t)+\cone\{(-1,1)\},&\text{if }t\leq 1/2\\
    (-3/2,-1/2)+\cone\{(-1,1)\},&\text{if }t>1/2\\
   \end{cases}
 \end{align*}
 \item $z_0\in ]x_3,x_4[\iff z_0=(t-1,-t),\,t\in]0,1[$. Replacing this in equations~\eqref{eq:bwp1} through~\eqref{eq:bwp4}, we obtain:
 \begin{align*}
 \eqref{eq:bwp1}&\qquad\text{implies}\qquad 3-2t\leq (t-2)u-tv,\\
 \eqref{eq:bwp2}&\qquad\text{implies}\qquad 2\leq (t-1)u-(t+1)v,\\
 \eqref{eq:bwp3}&\qquad\text{implies}\qquad 0\leq u-v,\\
 \eqref{eq:bwp4}&\qquad\text{implies}\qquad 0\leq -u+v.
 \end{align*}
 That is $u=v$, $v\leq -1$, $v\leq t-3/2$. Therefore
 \begin{align*}
 T_5(t-1,-t)&=\{(u,v)\::\:u=v,\,v\leq \min\{-1,t-3/2\}\}\\
 &=\begin{cases}
    \{(u,v)\::\:u=v,\,v\leq t-3/2\},&\text{if }t\leq 1/2\\
    \{(u,v)\::\:u=v,\,v\leq -1\},&\text{if }t>1/2
   \end{cases}\\
 &=\begin{cases}
    (t-3/2,t-3/2)+\cone\{(-1,-1)\},&\text{if }t\leq 1/2\\
    (-1,-1)+\cone\{(-1,-1)\},&\text{if }t>1/2\\
   \end{cases}
 \end{align*}
  \item $z_0\in ]x_4,x_1[\iff z_0=(t,t-1),\,t\in]0,1[$. Replacing this in equations~\eqref{eq:bwp1} through~\eqref{eq:bwp4}, we obtain:
 \begin{align*}
 \eqref{eq:bwp1}&\qquad\text{implies}\qquad 1\leq -u-v,\\
 \eqref{eq:bwp2}&\qquad\text{implies}\qquad 2-3t\leq tu+(t-2)v,\\
 \eqref{eq:bwp3}&\qquad\text{implies}\qquad -t\leq (t+1)u+(t-1)v,\\
 \eqref{eq:bwp4}&\qquad\text{implies}\qquad -1\leq u+v.
 \end{align*}
 That is $u+v=-1$, $v\leq t-1$, $v\leq -1/2$. Therefore
 \begin{align*}
 T_5(t,t-1)&=\{(u,v)\::\:u+v=-1,\,v\leq \min\{-1/2,t-1\}\}\\
 &=\begin{cases}
    \{(u,v)\::\:u+v=-1,\,v\leq t-1\},&\text{if }t\leq 1/2\\
    \{(u,v)\::\:u+v=-1,\,v\leq -1/2\},&\text{if }t>1/2
   \end{cases}\\
 &=\begin{cases}
    (-t,t-1)+\cone\{(1,-1)\},&\text{if }t\leq 1/2\\
    (-1/2,-1/2)+\cone\{(1,-1)\},&\text{if }t>1/2\\
   \end{cases}
 \end{align*}
 \end{itemize}
Altogether, we obtain the following correspondence rule:
\begin{align*}
 T_5(1,0)&=\co\{(-1/2,-1/2),(-1/2,1/2)\}+\cone\co\{(1,-1),(1,1)\},\\
T_5(0,1)&=\co\{(-1,0),(-2,0)\}+\cone\co\{(1,1),(-1,1)\},\\
T_5(-1,0)&=\co\{(-3/2,-1/2),(-3/2,-3/2)\}+\cone\co\{(-1,1),(-1,-1)\},\\
T_5(0,-1)&=\co\{(-1,-1),(0,-1)\}+\cone\co\{(-1,-1),(1,-1)\},\\
T_5(0,0)&=(-1,-1/2),
\end{align*}
and, for all $t\in]0,1[$,
\begin{align*}
 T_5(1-t,t)&=\begin{cases}
    (-t-1/2,-t+1/2)+\cone\{(1,1)\},&\text{if }t\leq 1/2\\
    (-1,0)+\cone\{(1,1)\},&\text{if }t>1/2\\
   \end{cases}\\
 T_5(-t,1-t)&=\begin{cases}
    (t-2,-t)+\cone\{(-1,1)\},&\text{if }t\leq 1/2\\
    (-3/2,-1/2)+\cone\{(-1,1)\},&\text{if }t>1/2\\
   \end{cases}\\
 T_5(t-1,-t)&=\begin{cases}
    (t-3/2,t-3/2)+\cone\{(-1,-1)\},&\text{if }t\leq 1/2\\
    (-1,-1)+\cone\{(-1,-1)\},&\text{if }t>1/2\\
   \end{cases}\\
 T_5(t,t-1)&=\begin{cases}
    (-t,t-1)+\cone\{(1,-1)\},&\text{if }t\leq 1/2\\
    (-1/2,-1/2)+\cone\{(1,-1)\},&\text{if }t>1/2\\
   \end{cases}
\end{align*}
See Figure~\ref{fig:bwp} for a graphical representation of the domain and range.

Note that the operator $T_5$ can be written as
\[
T_5=\left[\bigcup_{i=1}^8\co\{(y_i,y_i^*),(y_{i+1},y_{i+1}^*)\}+N_C\right]\cup\{((0,0),(-1,-1/2))\}.
\]
From Lemma~\ref{lem:orto}, the operator $S=\ds\bigcup_{i=1}^8\co\{(y_i,y_i^*),(y_{i+1},y_{i+1}^*)\}+N_C$ is 2-cyclically monotone. Also from Lemma~\ref{lem:orto}, $\mathcal{F}^{\mu_p}=S^{\mu_p}$. Since $((0,0),(-1,-1/2))\in\mathcal{F}^{\mu_p}$, we conclude that $T_5$ is also 2-cyclically monotone. Therefore $T_5$ is maximal 2-cyclically monotone.
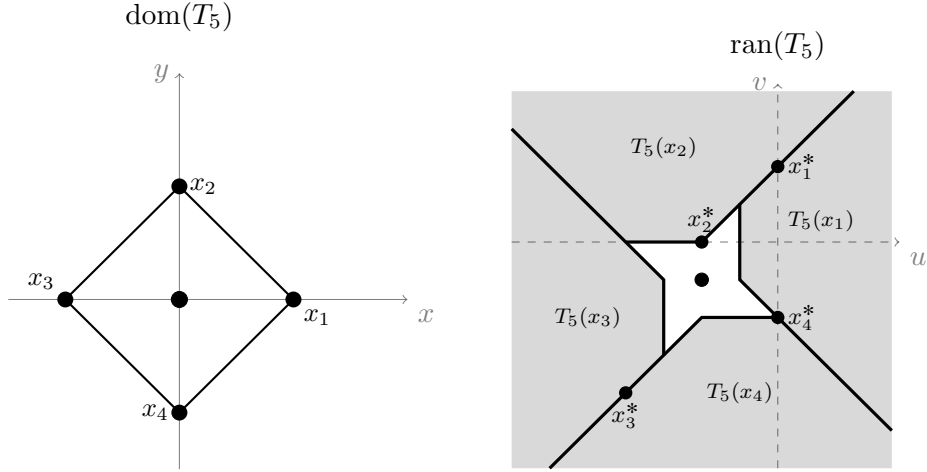
\begin{figure}
\begin{center}
\begin{tikzpicture}[scale=1.5]
\node at (0,2.5) {$\dom(T_5)$};

\draw[semitransparent,->](-1.5,0)--(2,0)node[below right]{$x$};
\draw[semitransparent,->](0,-1.5)--(0,2)node[left]{$y$};

\draw[thick] (1,0) -- (0,1) -- (-1,0) -- (0,-1) -- cycle;
\draw[fill] (0,0) circle (2pt);

\fill
(1,0)circle(2pt)node[below right]{\small $x_1$};
\fill
(-1,0)circle(2pt)node[above left]{\small $x_3$};
\fill
(0,1)circle(2pt)node[right]{\small $x_2$};
\fill
(0,-1)circle(2pt)node[left]{\small $x_4$};
\end{tikzpicture}
\qquad
\begin{tikzpicture}
\node at (0,2.6) {$\ran(T_5)$}; 

\fill[gray!30]
(1.5,2) -- (1,2) --  (-0.5,0.5) -- (-0.5,-0.5) -- (1.5,-2.5) -- cycle;
\fill[gray!30]
(1,2) --  (-1,0) --  (-2,0) -- (-3.5,1.5) -- (-3.5,2) -- cycle;
\fill[gray!30]
(-3.5,1.5) -- (-1.5,-0.5) -- (-1.5,-1.5) -- (-3,-3) -- (-3.5,-3) -- cycle;
\fill[gray!30]
(-3,-3) -- (-1,-1) -- (0,-1) -- (1.5,-2.5) -- (1.5,-3) -- cycle;

\draw[very thick] (-3.5,1.5) -- (-1.5,-0.5) -- (-1.5,-1.5) -- (-3,-3);
\draw[very thick] (1,2) --  (-0.5,0.5) -- (-0.5,-0.5) -- (1.5,-2.5);

\draw[very thick] (-1.5,-1.5) -- (-1,-1) -- (0,-1);
\draw[very thick] (-0.5,0.5) --  (-1,0) --  (-2,0);
%
%
\draw[fill] (-1,-0.5) circle (2.5pt);

\draw[dashed,semitransparent,->](-3.5,0)--(1.6,0)node[below right]{$u$};
\draw[dashed,semitransparent,->](0,-3)--(0,2.1)node[left]{$v$};

\fill
(0,1)circle(2.5pt)node[right]{\footnotesize $x_1^*$};
\fill
(-2,-2)circle(2.5pt)node[below]{\footnotesize $x_3^*$};
\fill
(-1,0)circle(2.5pt)node[above]{\footnotesize $x_2^*$};
\fill
(0,-1)circle(2.5pt)node[right]{\footnotesize $x_4^*$};

\node[above right] 
at (0,0) {\scriptsize $T_5(x_1)$};
\node 
at (-1.5,1.25) {\scriptsize $T_5(x_2)$};
\node 
at (-2.5,-1) {\scriptsize $T_5(x_3)$};
\node 
at (-0.5,-2) {\scriptsize $T_5(x_4)$};

\end{tikzpicture}
\end{center}
\caption{Perturbation of Bauschke and Wang example}\label{fig:bwp}
\end{figure}
\begin{remark}
It is possible to verify that if we use $x_5=(0,0)$, $x_5^*=(-1,-1/2)$ as an additional starting point, we would have obtained the same operator. 
\end{remark}

\input{proof2m2}

\end{document}

%% file: proof2m2.tex
\subsection{Another bizarre maximal 2-cyclically monotone operator}
Consider the starting points:
\begin{align*}
x_1&=(0,0)&x_1^*&=(-1,-1/2),\\
x_2&=(1, 0)&x_2^*&=(0,1), \\
x_3&=(0, 1)&x_3^*&=(-1, 0),\\
x_4&=(-1,0)&x_4^*&=(-2, 2),\\
x_5&=(0,-1)&x_5^*&=(0,-1).
\end{align*}
Using our procedure, we deduce the following images of $x_1,\ldots,x_5$:
\begin{align*}
T_5(x_1)&=\{-1\}\times[-1,0],\\
        &=\co\{(-1,-1),(-1,0)\},\\
T_5(x_2)&=\{(u,v)\::\:u\geq 0,\,-1-u\leq v\leq 1+u\},\\
        &=\co\{(0,-1),(0,1)\}+\cone\{(1,-1),(1,1)\},\\
T_5(x_3)&=\{(u,v)\::\:v\geq 0,\,-2-v\leq u\leq v-1\},\\
        &=\co\{(-1,0),(-2,0)\}+\cone\{(1,1),(-1,1)\},\\
T_5(x_4)&=\{(u,v)\::\:u\leq -2,\,u\leq v\leq -2-u\},\\
        &=\co\{(-2,0),(-2,-2)\}+\cone\{(-1,1),(-1,-1)\},\\
T_5(x_5)&=\{(u,v)\::\:v\leq -1,\,v\leq u\leq -1-v\},\\
        &=\co\{(-1,-1),(0,-1)\}+\cone\{(-1,-1),(1,-1)\}.
\end{align*}
Thus, $\mathcal{F}=\{(y_i,y_i^*)\}_{i=1}^{10}$ is given by
\begin{align*}
y_1&=(0,0)&y_1^*&=(-1,-1),&y_2&=(0, 0)&y_2^*&=(-1,0), \\
y_3&=(1,0)&y_3^*&=(0,-1),&y_4&=(1,0)&y_4^*&=(0,1),\\
y_5&=(0,1)&y_5^*&=(-1,0),&y_6&=(0,1)&y_6^*&=(-2,0),\\
y_7&=(-1,0)&y_7^*&=(-2,0),&y_8&=(-1,0)&y_8^*&=(-2,-2),\\
y_9&=(0,-1)&y_9^*&=(-1,-1),&y_{10}&=(0,-1)&y_{10}^*&=(0,-1),\\
\end{align*}
that is
\begin{align*}
\dom(\mathcal{F})&=\{(0,0),(1,0),(0,1),(-1,0),(0,-1)\}\\
\ran(\mathcal{F})&=\{(-1,-1),(-1,0),(0,-1),(0,1),(-2,0),(-2,-2)\}. 
\end{align*}
We now consider $z_0=(x,y)$ and $z_0^*=(u,v)$, and follow the steps of the proof of Proposition~\ref{pro:opfinite}, so we obtain the set of inequalities:
\begin{equation}\label{eq:bwb}
\begin{aligned}
 \max\left\{\begin{array}{c}-2x,-x,-2x-2y-1,\\-x-y,y-1,-y\end{array}\right\}&\leq xu+yv\\
 \max\left\{\begin{array}{c}1-2x,1-x,-x-y,\\-y,y,-2x-2y-2\end{array}\right\}&\leq (x-1)u+yv\\
 \max\left\{\begin{array}{c}-2x,-x,-y-1,-x-y,\\y-1,-2x-2y\end{array}\right\}&\leq xu+(y-1)v\\
 \max\left\{\begin{array}{c}-x-2,-2x-2,-y-1,y-4,\\-x-y-1,-2x-2y-2\end{array}\right\}&\leq (x+1)u+yv\\
 \max\left\{\begin{array}{c}-x-1,-2x-2,-y-1,y-1,\\-x-y-1,-2x-2y-2\end{array}\right\}&\leq xu+(y+1)v.
\end{aligned}
\end{equation}
Equations~\eqref{eq:bwb}, along with the domain conditions:
\[
x-1\leq y\leq x+1,\qquad -x-1\leq y \leq 1-x.
\]
allow us to obtain the simplified equations:
\begin{gather}
 \widetilde{M}(z_0,x_1):=\max\{-2x,-x,-x-y,-y\}\leq xu+yv,\label{eq:bwb1}\\
 \widetilde{M}(z_0,x_2):=\max\{1-2x,1-x\}\leq (x-1)u+yv,\label{eq:bwb2}\\
 \widetilde{M}(z_0,x_3):=\max\{-2x,-x,-x-y,-2x-2y\}\leq xu+(y-1)v,\label{eq:bwb3}\\
 \widetilde{M}(z_0,x_4):=\max\{-1-y,-1-x-y\}\leq (x+1)u+yv,\label{eq:bwb4}\\
  \widetilde{M}(z_0,x_5):=\max\{-1-x,-1-y,-1-x-y,-1+y\}\leq xu+(y+1)v.\label{eq:bwb5}
\end{gather}

To determine the domain of $T_5$, let $z_0\in C$ and write
\[
z_0=\sum_{i=1}^5\alpha_i x_i= (\alpha_2-\alpha_4,\alpha_3-\alpha_5),
\]
with $\alpha_i\geq 0$ and $\ds\sum_{i=1}^5\alpha_i=1$. 
The determination of $\dom(T_5)$ will be done in several steps.
\begin{itemize}
 \item Let $z_0\in\co\{x_1,x_2,x_3\}$, that is, $\alpha_4=\alpha_5=0$, $\alpha_1+\alpha_2+\alpha_3=1$, $\alpha_1,\alpha_2,\alpha_3\geq 0$ and $z_0=(\alpha_2,\alpha_3)$. Thus
 \begin{align*}
  \sum_{i=1}^5\alpha_i \widetilde{M}(z_0,x_i)&=(1-\alpha_2-\alpha_3)\max\{-2\alpha_2,-\alpha_2,-\alpha_2-\alpha_3,-\alpha_3\}\\
  \phantom{=}+\alpha_2\max\{1-&2\alpha_2,1-\alpha_2\}+\alpha_3\max\{-2\alpha_2,-\alpha_2,-\alpha_2-\alpha_3,-2\alpha_2-2\alpha_3\}\\
  &=(1-\alpha_2-\alpha_3)\max\{-\alpha_2,-\alpha_3\}+\alpha_2(1-\alpha_2)+\alpha_3(-\alpha_2)\\
  &=(1-\alpha_2-\alpha_3)(\max\{-\alpha_2,-\alpha_3\}+\alpha_2)\\
  &=\begin{cases}
    0&\text{if }\alpha_2\leq \alpha_3,\\
    (\alpha_2-\alpha_3)(1-\alpha_2-\alpha_3)&\text{if }\alpha_2>\alpha_3.\\
    \end{cases}
 \end{align*}
We conclude that $\ds\sum_{i=1}^5\alpha_i \widetilde{M}(z_0,x_i)>0$ if, and only if, $\alpha_2>\alpha_3$ and $1-\alpha_2-\alpha_3>0$. Therefore, all $z_0=(x,y)$ such that $x>y\geq 0$ and $x+y<1$ do not belong to $\dom(T_6)$.
 \item Let $z_0\in\co\{x_1,x_3,x_4\}$, that is, $\alpha_2=\alpha_5=0$, $\alpha_1+\alpha_3+\alpha_4=1$, $\alpha_1,\alpha_3,\alpha_4\geq 0$ and $z_0=(-\alpha_4,\alpha_3)$. Thus
 \begin{align*}
  \sum_{i=1}^5\alpha_i \widetilde{M}(z_0,x_i)&=(1-\alpha_3-\alpha_4)\max\{2\alpha_4,\alpha_4,\alpha_4-\alpha_3,-\alpha_3\}\\
  &\phantom{=}+\alpha_3\max\{2\alpha_4,\alpha_4,\alpha_4-\alpha_3,2\alpha_4-2\alpha_3\}\\
  &\phantom{=}+\alpha_4\max\{-1-\alpha_3,-1+\alpha_4-\alpha_3\}\\
  &=(1-\alpha_3-\alpha_4)(2\alpha_4)+\alpha_3(2\alpha_4)+\alpha_4(-1+\alpha_4-\alpha_3)\\
  &=\alpha_4(1-\alpha_4-\alpha_3)
 \end{align*}
Therefore $\ds\sum_{i=1}^5\alpha_i \widetilde{M}(z_0,x_i)>0$ if, and only if, $\alpha_4>0$ and $\alpha_3+\alpha_4<1$. Therefore, all $z_0=(x,y)$ such that $x<0$ and $y<1+x$ do not belong to $\dom(T_6)$.
 \item Let $z_0\in\co\{x_1,x_4,x_5\}$, that is, $\alpha_2=\alpha_3=0$, $\alpha_1+\alpha_4+\alpha_5=1$, $\alpha_1,\alpha_4,\alpha_5\geq 0$ and $z_0=(-\alpha_4,-\alpha_5)$. Thus
 \begin{align*}
  \sum_{i=1}^5\alpha_i \widetilde{M}(z_0,x_i)&=(1-\alpha_4-\alpha_5)\max\{2\alpha_4,\alpha_4,\alpha_4+\alpha_5,\alpha_5\}\\
  &\phantom{=}+\alpha_4\max\{-1+\alpha_5,-1+\alpha_4+\alpha_5\}\\
  &\phantom{=}+\alpha_5\max\{-1+\alpha_4,-1+\alpha_5,-1+\alpha_4+\alpha_5,-1-\alpha_5\}\\
  &=(1-\alpha_4-\alpha_5)(\alpha_4+\max\{\alpha_4,\alpha_5\})+\alpha_4(-1+\alpha_4+\alpha_5)\\
  &\phantom{=}+\alpha_5(-1+\alpha_4+\alpha_5)\\
  &=(1-\alpha_4-\alpha_5)(\max\{\alpha_4,\alpha_5\}-\alpha_5)\\
  &=\begin{cases}
    0&\text{if }\alpha_4\leq \alpha_5,\\
    (\alpha_4-\alpha_5)(1-\alpha_4-\alpha_5)&\text{if }\alpha_4>\alpha_5.\\
    \end{cases}
 \end{align*}
We conclude that $\ds\sum_{i=1}^5\alpha_i \widetilde{M}(z_0,x_i)>0$ if, and only if, $\alpha_4>\alpha_5$ and $1-\alpha_4-\alpha_5>0$. Therefore, all $z_0=(x,y)$ such that $x<y\leq 0$ and $x+y>-1$ do not belong to $\dom(T_6)$.
 \item Let $z_0\in\co\{x_1,x_2,x_5\}$, that is, $\alpha_3=\alpha_4=0$, $\alpha_1+\alpha_2+\alpha_5=1$, $\alpha_1,\alpha_2,\alpha_5\geq 0$ and $z_0=(\alpha_2,-\alpha_5)$. Thus
 \begin{align*}
  \sum_{i=1}^5\alpha_i \widetilde{M}(z_0,x_i)&=(1-\alpha_2-\alpha_5)\max\{-2\alpha_2,-\alpha_2,-\alpha_2+\alpha_5,\alpha_5\}\\
  &\phantom{=}+\alpha_2\max\{1-2\alpha_2,1-\alpha_2\}\\
  &\phantom{=}+\alpha_5\max\{-1-\alpha_2,-1+\alpha_5,-1-\alpha_2+\alpha_5,-1-\alpha_5\}\\
  &=(1-\alpha_2-\alpha_5)\alpha_5+\alpha_2(1-\alpha_2)+\alpha_5(-1+\alpha_5)\\
  &=\alpha_2(1-\alpha_2-\alpha_5)
 \end{align*}
Therefore $\ds\sum_{i=1}^5\alpha_i \widetilde{M}(z_0,x_i)>0$ if, and only if, $\alpha_2>0$ and $\alpha_2+\alpha_5<1$. Therefore, all $z_0=(x,y)$ such that $x>0$ and $y>x-1$ do not belong to $\dom(T_6)$.
\end{itemize}

Up to now, we have proved that $\dom(T_6)$ is contained in
\[
\bigcup_{i=2}^5[x_i,x_{i+1}]\cup \{(x,y)\::\:|x|\leq |y|\leq 1-|x|,\,xy\geq 0\}.
\]
considering $x_6=x_2$. Next, we will prove that this set is exactly $\dom(T_6)$ and also find the correspondence rule for $T_6$.
\begin{itemize}
 \item Let $z_0\in]x_3,x_5[\setminus\{x_1\}\iff z_0=(0,1-2t),\,t\in]0,1[,t\neq 1/2$. Thus
 \begin{align*}
 \eqref{eq:bwb1}&\quad\text{implies}\quad\max\{0,2t-1\}\leq (1-2t)v,\\
 \eqref{eq:bwb2}&\quad\text{implies}\quad 1\leq -u+(1-2t)v,\\
 \eqref{eq:bwb3}&\quad\text{implies}\quad\max\{0,4t-2\}\leq -2tv,\\
  \eqref{eq:bwb4}&\quad\text{implies}\quad 2t-2\leq u+(1-2t)v,\\
 \eqref{eq:bwb5}&\quad\text{implies}\quad\max\{2t-2,-2t\}\leq 2(1-t)v.
\end{align*}
If $t<1/2$, 
these equations turn into
$v=0$, $2t-2\leq u\leq -1$ so
\[
T_6(0,1-2t)=[2t-2,-1]\times\{0\}.
\]
Now, if $t>1/2$, we obtain
$v=-1$, $-1\leq u\leq 2t-2$, so 
\[
T_6(0,1-2t)=[-1,2t-2]\times\{-1\}.
\]
\item Let $z_0\in ]x_2,x_3[\iff z_0=(1-t,t),\,t\in]0,1[$. Thus
 \begin{align*}
 \eqref{eq:bwb1}&\quad\text{implies}\quad\max\{t-1,-t\}\leq (1-t)u+tv,\\
 \eqref{eq:bwb2}&\quad\text{implies}\quad 1\leq v-u,\\
 \eqref{eq:bwb3}&\quad\text{implies}\quad -1\leq u-v,\\
  \eqref{eq:bwb4}&\quad\text{implies}\quad -1-t\leq (2-t)u+tv,\\
 \eqref{eq:bwb5}&\quad\text{implies}\quad t-1\leq (1-t)u+(t+1)v.
\end{align*}
Hence $v=1+u$, $u\geq \max\{-2t,-1\}$. Therefore
\begin{align*}
T_6(1-t,t)&=
\begin{cases}
\{(u,v)\::\:v=1+u,\,u\geq -2t\},&\text{if }t\leq 1/2,\\
\{(u,v)\::\:v=1+u,\,u\geq -1\},&\text{if }t> 1/2,
\end{cases}\\
&=
\begin{cases}
(-2t,1-2t)+\cone\{(1,1)\},&\text{if }t\leq 1/2,\\
(-1,0)+\cone\{(1,1)\},&\text{if }t> 1/2,
\end{cases}
\end{align*}
\item Let $z_0\in ]x_3,x_4[\iff z_0=(-t,1-t),\,t\in]0,1[$. Thus
 \begin{align*}
 \eqref{eq:bwb1}&\quad\text{implies}\quad 2t \leq -tu+(1-t)v,\\
 \eqref{eq:bwb2}&\quad\text{implies}\quad 1+2t\leq -(t+1)u+(1-t)v,\\
 \eqref{eq:bwb3}&\quad\text{implies}\quad 2\leq -u-v,\\
  \eqref{eq:bwb4}&\quad\text{implies}\quad -2\leq u+v,\\
 \eqref{eq:bwb5}&\quad\text{implies}\quad\max\{-1+t,-t\}\leq -tu+(2-t)v.
\end{align*}
Hence $u+v=-2$, $u\leq -2$. Therefore
\[
 T_6(-t,1-t)=\{(u,v)\::\:u+v=-2,\,u\leq -2\}=(-2,0)+\cone\{(-1,1)\}.
\]
\item Let $z_0\in ]x_4,x_5[\iff z_0=(t-1,-t),\,t\in]0,1[$. Thus
 \begin{align*}
 \eqref{eq:bwb1}&\quad\text{implies}\quad\max\{2-2t,1\}\leq (t-1)u-tv,\\
 \eqref{eq:bwb2}&\quad\text{implies}\quad 3-2t\leq (t-2)u-tv,\\
 \eqref{eq:bwb3}&\quad\text{implies}\quad 2\leq (t-1)u-(t+1)v,\\
  \eqref{eq:bwb4}&\quad\text{implies}\quad 0\leq u-v,\\
 \eqref{eq:bwb5}&\quad\text{implies}\quad 0\leq -u+v.
\end{align*}
$u\leq \min\{2t-2,-1\}$

Hence $v=u$, $u\leq \min\{2t-2,-1\}$. Therefore
\begin{align*}
T_6(t-1,-t)&=
\begin{cases}
\{(u,v)\::\:v=u,\,u\leq 2t-2\},&\text{if }t\leq 1/2,\\
\{(u,v)\::\:v=u,\,u\leq -1\},&\text{if }t> 1/2,
\end{cases}\\
&=
\begin{cases}
(2t-2,2t-2)+\cone\{(-1,-1)\},&\text{if }t\leq 1/2,\\
(-1,-1)+\cone\{(-1,-1)\},&\text{if }t> 1/2,
\end{cases}
\end{align*}
\item Let $z_0\in ]x_5,x_2[\iff z_0=(t,t-1),\,t\in]0,1[$. Thus
 \begin{align*}
 \eqref{eq:bwb1}&\quad\text{implies}\quad 1-t\leq tu+(t-1)v,\\
 \eqref{eq:bwb2}&\quad\text{implies}\quad 1\leq -u-v,\\
 \eqref{eq:bwb3}&\quad\text{implies}\quad\max\{-2t+1,-4t+2\}\leq tu+(t-2)v,\\
  \eqref{eq:bwb4}&\quad\text{implies}\quad -t\leq (t+1)u+(t-1)v,\\
 \eqref{eq:bwb5}&\quad\text{implies}\quad -1\leq u+v.
\end{align*}
Hence $u+v=-1$, $v\leq -1$. Therefore
\[
 T_6(t,t-1)=\{(u,v)\::\:u+v=-1,\,v\leq -1\}=(0,-1)+\cone\{(1,-1)\}.
\]
\item Let $z=(x,y)$ with $0<x\leq y<1-x$. Then
 \begin{align}
 \eqref{eq:bwb1}&\quad\text{implies}\quad -x\leq xu+yv,\label{eq:bwb1n}\\
 \eqref{eq:bwb2}&\quad\text{implies}\quad 1-x\leq (x-1)u+yv,\label{eq:bwb2n}\\
 \eqref{eq:bwb3}&\quad\text{implies}\quad -x\leq xu+(y-1)v,\label{eq:bwb3n}\\
  \eqref{eq:bwb4}&\quad\text{implies}\quad -1-y\leq (x+1)u+yv,\label{eq:bwb4n}\\
 \eqref{eq:bwb5}&\quad\text{implies}\quad -1+y\leq xu+(y+1)v.\label{eq:bwb5n}
\end{align}
Combining~\eqref{eq:bwb1n} and~\eqref{eq:bwb3n}, we obtain
\[
-\dfrac{x}{y}(u+1)\leq v\leq \dfrac{x}{1-y}(u+1),
\]
so $0\leq \dfrac{x(u+1)}{(1-y)y}$, which implies $u\geq -1$. On the other hand, combining~\eqref{eq:bwb2n} and~\eqref{eq:bwb3n}, we obtain 
\[
\dfrac{1-x}{y}(u+1)\leq v\leq \dfrac{x}{1-y}(u+1),
\]
so $0\leq \dfrac{(x+y-1)x(u+1)}{(1-y)y}$, which implies $u\leq -1$. Therefore $u=-1$ and, after replacing in any of the previous inequalities, $v=0$. It is straightforward to verify that $(-1,0)$ also satisfies~\eqref{eq:bwb4n} and~\eqref{eq:bwb5n}. Thus, for $0<x\leq y<1-x$,
\[
 T_6(x,y)=(-1,0).
\]
\item Let $z=(x,y)$ with $-1-x<y\leq x<0$. Then
\begin{align}
\eqref{eq:bwb1}&\quad\text{implies}\quad -x-y\leq xu+yv,\label{eq:bwb1nn}\\
 \eqref{eq:bwb2}&\quad\text{implies}\quad 1-2x\leq (x-1)u+yv,\label{eq:bwb2nn}\\
 \eqref{eq:bwb3}&\quad\text{implies}\quad -2x-2y\leq xu+(y-1)v,\label{eq:bwb3nn}\\
  \eqref{eq:bwb4}&\quad\text{implies}\quad -1-x-y\leq (x+1)u+yv,\label{eq:bwb4nn}\\
 \eqref{eq:bwb5}&\quad\text{implies}\quad -1-x-y\leq xu+(y+1)v.\label{eq:bwb5nn}
\end{align}
Combining~\eqref{eq:bwb1nn} and~\eqref{eq:bwb5nn}, we obtain
\[
-\dfrac{x}{y+1}(u+1)\leq v+1\leq -\dfrac{x}{y}(u+1)
\]
so $\dfrac{x(u+1)}{y(y+1)}\leq 0$, which implies $u\leq -1$. On the other hand, combining~\eqref{eq:bwb4nn} and~\eqref{eq:bwb5nn}, we obtain 
\[
-\dfrac{x}{y+1}(u+1)\leq v+1\leq -\dfrac{x+1}{y}(u+1)
\]
so $\dfrac{(x+y+1)(u+1)}{y(y+1)}\leq 0$, which implies $u\geq -1$. Therefore $u=-1$ and, after replacing in any of the previous inequalities, $v=-1$. It is also straightforward to verify that $(-1,-1)$ also satisfies~\eqref{eq:bwb2n} and~\eqref{eq:bwb3n}.Thus, for $-1-x<y\leq x<0$,
\[
 T_6(x,y)=(-1,-1).
\]
\end{itemize}
Therefore, the full correspondence rule of $T_6$ is
\begin{align*}
T_6(0,0)&=\co\{(-1,-1),(-1,0)\},\\
T_6(1,0)&=\co\{(0,-1),(0,1)\}+\cone\{(1,-1),(1,1)\},\\
T_6(0,1)&=\co\{(-1,0),(-2,0)\}+\cone\{(1,1),(-1,1)\},\\
T_6(-1,0)&=\co\{(-2,0),(-2,-2)\}+\cone\{(-1,1),(-1,-1)\},\\
T_6(0,-1)&=\co\{(-1,-1),(0,-1)\}+\cone\{(-1,-1),(1,-1)\}.
\end{align*}
and, for all $t\in]0,1[$,
\begin{align*}
T_6(1-t,t)&=\begin{cases}
          (-2t,1-2t),&\text{if }t\in]0,1/2[,\\ 
          (-1,0),&\text{if }t\in[1/2,1[,
         \end{cases}\\
T_6(-t,1-t)&=(-2,0),\\
T_6(-1+t,-t)&=\begin{cases}
          (2t-2,2t-2),&\text{if }t\in]0,1/2[,\\ 
          (-1,-1),&\text{if }t\in[1/2,1[,
         \end{cases}\\
T_6(t,t-1)&=(0,-1).
\end{align*}
In addition, for $\alpha\in]-1,1[\setminus\{0\}$,
\[
T_6(0,\alpha)=
\begin{cases}
[-\alpha-1,-1]\times\{0\}&\text{if }\alpha\in]0,1[,\\
[-1,-\alpha-1]\times\{-1\}&\text{if }\alpha\in]-1,0[,
\end{cases}
\]
and 
\[
T_6(x,y)=
\begin{cases}
 (-1,0),&\text{if }0<x\leq y<1-x,\\
 (-1,-1),&\text{if }-1-x<y\leq x<0.
\end{cases}
\]
See Figure~\ref{fig:bwb} for a partial graph of the domain and range of $T_6$.
\begin{figure}\centering
\begin{tikzpicture}[scale=1.5]
\node at (0,2.5) {$\dom(T_6)$};


\draw[semitransparent,->](-1.5,0)--(2,0)node[below right]{$x$};
\draw[semitransparent,->](0,-1.5)--(0,2)node[left]{$y$};

\fill[black!70] 
(0.5,0.5) -- (0,1) -- (0,0)--cycle;
\fill[black!70] 
(-0.5,-0.5) -- (0,-1) -- (0,0)--cycle;
\draw[very thick] (1,0) -- (0,1) -- (-1,0) -- (0,-1) -- cycle;
\draw[very thick] (0.5,0.5) -- (-0.5,-0.5);
\draw[very thick] (0.5,0.5) -- (-0.5,-0.5);
\draw[very thick] (0,1) -- (0,-1);

\fill
(0,0)circle(1.5pt)node[above left]{\small $x_1$};
\fill
(1,0)circle(1.5pt)node[below right]{\small $x_2$};
\fill
(0,1)circle(1.5pt)node[above right]{\small $x_3$};
\fill
(-1,0)circle(1.5pt)node[below]{\small $x_4$};
\fill
(0,-1)circle(1.5pt)node[left]{\small $x_5$};
\end{tikzpicture}
\quad
\begin{tikzpicture}[scale=1]
\node at (-1,2.6) {$\ran(T_6)$}; 

\fill[gray!30] 
(1.5,2) -- (1,2) --  (0,1) -- (0,-1) -- (1.5,-2.5) --cycle;
\fill[gray!30] 
(1,2) --  (-1,0) --  (-2,0) -- (-3.5,1.5) -- (-3.5,2) -- cycle;
\fill[gray!30] 
(-3.5,1.5) -- (-2,0) -- (-2,-2) -- (-3,-3) -- (-3.5,-3) -- cycle;
\fill[gray!30] 
(-3,-3) -- (-1,-1) -- (0,-1) -- (1.5,-2.5) -- (1.5,-3)-- cycle;
\draw[line width=2pt] (-1,0) -- (-1,-1);
\draw[very thick] (1,2) --  (0,1) -- (0,-1) -- (1.5,-2.5);
\draw[very thick] (0,1) --  (-1,0) --  (-2,0) -- (-3.5,1.5);
\draw[very thick] (-2,0) -- (-2,-2) -- (-3,-3);
\draw[very thick] (-3,-3) -- (-1,-1) -- (0,-1);


\draw[dashed,semitransparent,->](-3.5,0)--(1.8,0)node[below right]{$u$};
\draw[dashed,semitransparent,->](0,-3)--(0,2.1)node[left]{$v$};

\fill
(-1,-0.5)circle(2pt)node[right]{\footnotesize $x_1^*$};
\fill
(0,1)circle(2pt)node[right]{\footnotesize $x_2^*$};
\fill
(-1,0)circle(2pt)node[above]{\footnotesize $x_3^*$};
\fill
(-2,-2)circle(2pt)node[below]{\footnotesize $x_4^*$};
\fill
(0,-1)circle(2pt)node[below left]{\footnotesize $x_5^*$};
\node[above] 
at (1,0) {\scriptsize $T_6(x_2)$};
\node 
at (-1.5,1.25) {\scriptsize $T_6(x_3)$};
\node 
at (-3,-1) {\scriptsize $T_6(x_4)$};
\node 
at (-0.5,-2) {\scriptsize $T_6(x_5)$};

\end{tikzpicture}
\caption{Yet another maximal 2-cyclically monotone operator}\label{fig:bwb}
\end{figure}
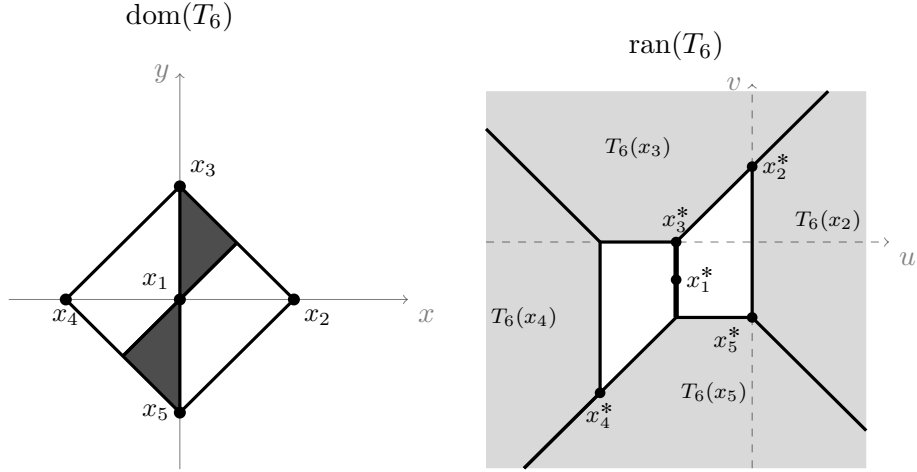

\subsubsection{Proof of maximal 2-cyclical monotonicity}
To prove the maximality of $T_6$, consider $T_6=\ds\bigcup_{i=1,\ldots,6}G_i$, where $G_i=T_6|_{D_i}$, $i=1,\ldots, 6$, and
\begin{gather*}
 D_1=\{(x,y)\in\R^2\::\:|x|+|y|=1\},\quad 
 D_2=\{(x,y)\in\R^2\::\:0<x\leq y<1-x\},\\
 D_3=\{(x,y)\in\R^2\::\:-1-x<y\leq x<0\},\quad
 D_4=\{(0,\alpha)\in\R^2\::\:0<\alpha<1\}\\
 D_5=\{(0,0)\},\quad 
 D_6=\{(0,\alpha)\in\R^2\::\:-1<\alpha<0\}.
\end{gather*}

Let $A=(a,a^*)$, $B=(b,b^*)$, $C=(c,c^*)$  in $T_6$, we need to prove that
\begin{equation}\label{eq:S}
S=\inner{b-a}{a^*}+\inner{c-b}{b^*}+\inner{a-c}{c^*}\leq 0.
\end{equation}
As $T_6=\displaystyle\bigcup_{i=1}^6 G_i$ then $A\in G_i$, $B\in G_j$, $C\in G_k$, where $(i,j,k)$ must belong to the set $\{1,\ldots,6\}^3$ which has 216 elements.

To prove the vast majority of the 216 cases, we are going to use the command \texttt{Simplify}~\cite{cite:simplify} of the symbolic calculus software \emph{Mathematica}. We used \emph{Mathematica} 8.0.4 for Linux. First consider the following definitions
\lstset{language=Mathematica}
\begin{lstlisting}
G2[p_, q_, r_, s_] := 
   (0 < p <= q < 1 - p) && (r == -1) && (s == 0) 
     && (q > 0) && (p - q <= 0)
     
G3[p_, q_, r_, s_] := 
   (-1 - p < q <= p < 0) && (r == -1) && (s == -1) 
     && (q < 0)
     
G4[p_, q_, r_, s_] := 
   (p == 0) && (0 < q < 1) && (-q - 1 <= r <= -1) 
     && (s == 0) && (r > -2)
     
G5[p_, q_, r_, s_] := 
   (p == 0) && (q == 0) && (r == -1) && (-1 <= s <= 0)

G6[p_, q_, r_, s_] := 
   (p == 0) && (-1 < q < 0) && (-1 <= r <= -1 - q) 
     && (s == -1) && (r < 0)
\end{lstlisting}
Note that $((p,q),(r,s))\in G_l$, for $l\in\{2,\ldots,6\}$, if, and only if, the respective evaluation of the boolean functions \lstinline!Gl[p,q,r,s]! is \lstinline!True!.

\paragraph{The case $(i,j,k)=(1,1,1)$:}
This case is equivalent to prove that $G_1$ is 2-cyclically monotone. In view of the definition of $T_6$, $G_1$ can be expressed as
\[
G_1=\bigcup_{i=1}^{10} [(w_i,w_i^*),(w_{i+1},w_{i+1}^*)]+N_{D_1},
\]
where $(w_{11},w_{11}^*)=(w_1,w_1^*)$ and 
\begin{align*}
w_1&=(1,0),&w_1^*&=(0,-1),&w_2&=(1, 0),&w_2^*&=(0,1), \\
w_3&=(1/2,1/2),&w_3^*&=(-1,0),&w_4&=(0,1),&w_4^*&=(-1,0),\\
w_5&=(0,1),&w_5^*&=(-2,0),&w_6&=(-1,0),&w_6^*&=(-2,0),\\
w_7&=(-1,0),&w_7^*&=(-2,-2),&w_8&=(-1/2,-1/2),&w_8^*&=(-1,-1),\\
w_9&=(0,-1),&w_9^*&=(-1,-1),&w_{10}&=(0,-1),&w_{10}^*&=(0,-1).\\
\end{align*}
Using the procedure \texttt{ispmono} given in Section~\ref{sec:testpmono}, we verify that $\mathcal{W}=\{(w_i,w_i^*)\}_{i=1}^{10}$ is 2-cyclically monotone. It is straightforward to verify that
\[
\inner{w_i-w_{i+1}}{w_i^*-w_{i+1}^*}=0,\quad\forall\,i=1,\ldots,10.
\]
Thus, using Lemma~\ref{lem:orto}, we conclude that $G_1$ is also 2-cyclically monotone and $G_1^{\mu_2}=\mathcal{W}^{\mu_2}$.

This proves 1 case of 216.

\paragraph{The case $(i,j,k)=(1,1,k)$, $k\neq 1$:} By cyclicity of $S$, this case is equivalent to the cases $(1,j,1)$ and $(i,1,1)$, $i,j\neq 1$, totalling 15 cases. 

Note that, $S\leq 0$, for all $A,B\in G_1$, $C\in G_k$ is equivalent to $G_k\subset G_1^{\mu_2}$. Since $G_1^{\mu_2}=\mathcal{W}^{\mu_2}$, we use the procedure \texttt{polareqs} given in Section~\ref{sec:maxima} applied to the finite operator $\mathcal{W}=\{(w_i,w_i^*)\}_{i=1}^{10}$. So any $\bar x=(x,y)$, $\bar x^*=(u,v)$ such that $(\bar x,\bar x^*)\in \mathcal{W}^{\mu_2}$ must satisfy the equations:
\begin{align*}
\max\left\{\begin{array}{c}-(x+2),-2(x+1),-(y+1)\\-(x+y+1),y-4,-2(x+y+1)\end{array}\right\}&\leq u(x+1)+vy,\\
\max\left\{\begin{array}{c}-2(x+1),-(2x+3)/2,-y-1,\\-y-x-1,y-2,-2(y+x+1)\end{array}\right\}&\leq u(x+1/2)+v(y+1/2),\\
\max\left\{\begin{array}{c}-x-1,-2(x+1),-y-1,\\-y-x-1,y-1,-2(y+x+1)\end{array}\right\}&\leq ux+v(y+1),\\
\max\left\{\begin{array}{c}-2x,-x,-y-1,-y-x,\\y-1,-2(y+x)\end{array}\right\}&\leq ux+v(y-1),\\
\max\left\{\begin{array}{c}-(2x-1)/2,-(4x-1)/2,-y-x,\\(2y-1)/2,-(2y+1)/2,-(4y+4x+1)/2\end{array}\right\}&\leq u(x-1/2)+v(y-1/2),\\
\max\left\{\begin{array}{c}1-2x,1-x,-y-x,\\-y,y,-2(y+x+1)\end{array}\right\}&\leq u(x-1)+vy,
\end{align*}
These equations can be implemented in \emph{Mathematica}, as boolean functions in terms of $(x,y,u,v)$.
\begin{lstlisting}
pol1[x_, y_, u_, v_] := 
   Max[-(x + 2), -2 (x + 1), -(y + 1), -(x + y + 1), 
     y - 4, -2 (x + y + 1)] <= u (x + 1) + v y
   
pol2[x_, y_, u_, v_] := 
   Max[-2 (x + 1), -(2 x + 3)/2, -y - 1, -y - x - 1, 
     y - 2, -2 (y + x + 1)] <= u (x + 1/2) + v (y + 1/2)
   
pol3[x_, y_, u_, v_] := 
   Max[-x - 1, -2 (x + 1), -y - 1, -y - x - 1, 
     y - 1, -2 (y + x + 1)] <= u x + v (y + 1)
   
pol4[x_, y_, u_, v_] := 
   Max[-2 x, -x, -y - 1, -y - x, y - 1, -2 (y + x)] 
     <= u x + v (y - 1)
   
pol5[x_, y_, u_, v_] := 
   Max[-(2 x - 1)/2, -(4 x - 1)/2, -y - x, (2 y - 1)/2, 
     -(2 y + 1)/2, -(4 y + 4 x + 1)/2] 
     <= u (x - 1/2) + v (y - 1/2)
   
pol6[x_, y_, u_, v_] := 
   Max[1 - 2 x, 1 - x, -y - x, -y, y, -2 (y + x + 1)]
     <= u (x - 1) + v y 
\end{lstlisting}
The instruction
\begin{lstlisting}
Table[
   Simplify[eq[x, y, u, v],g[x, y, u, v]],
     {eq, {pol1, pol2, pol3, pol4, pol5, pol6}}, 
     {g, {G2, G3, G4, G5, G6}}]
\end{lstlisting}
allow us to obtain a $6\times 5$ array full of \lstinline!True! symbols. This implies that for $i\in\{2,\ldots,6\}$, every element in $G_i$ satisfies the defining equations of $G_1^{\mu_2}$. Thus $G_i\subset G_1^{\mu_2}$.

\paragraph{The case $(i,j,k)=(1,j,k)$, $j,k\neq 1$:}
By cyclicity of $S$, this case is equivalent to the cases $(i,1,k)$ and $(i,j,1)$. Each one comprehends 25 cases, totalling 75 cases.

For fixed $(b,b^*)$, $(c,c^*)$, consider the supremum
\begin{align*}
\hat S&=\sup_{(a,a^*)\in G_1}\inner{b-a}{a^*}+\inner{a}{c^*}\\
 &=\sup\left\{\inner{b-a_l(t)}{a_l^*(t)}+\inner{a_l(t)}{c^*}\::\:\begin{array}{c}l=1,\ldots,10,\\t\in[0,1]\end{array}\right\},
\end{align*}
where $(a_l(t),a_l^*(t))=(1-t)(w_l,w_l^*)+t(w_{l+1},w_{l+1}^*)$, for all $l=1,\ldots,10$.
As $G_1$ is compact, the previous supremum is attained, so there exist $\tilde l\in\{1,\ldots,10\}$, $s\in [0,1]$ such that
\[
\hat S=\inner{b-a_{\tilde l}(s)}{a_{\tilde l}^*(s)}+\inner{a_{\tilde l}(s)}{c^*}.
\]
Since $t\mapsto\inner{b-a_{\tilde l}(t)}{a_{\tilde l}^*(t)}+\inner{a_{\tilde l}(t)}{c^*}$ is affine, the supremum must be attained when $s=0$ or $s=1$. This means 
\[
\hat S=\inner{b-w_k}{w_k^*}+\inner{w_k}{c^*}
\]
where $k={\tilde l}$ or $k={\tilde l}+1$. Therefore
\[
\hat S=\max_{l=1,\ldots,10} \inner{b-w_l}{w_l^*}+\inner{w_l}{c^*}
\]
and thus
\begin{gather*}
\inner{b-a}{a^*}+\inner{c-b}{b^*}+\inner{a-c}{c^*}\leq 0,\quad\forall (a,a^*)\in G_1,\,(b,b^*)\in G_j,\,(c,c^*)\in G_k\\
\text{if and only if}\\
\inner{b-w_i}{w_i^*}+\inner{c-b}{b^*}+\inner{w_i-c}{c^*}\leq 0,\quad
\begin{array}{c}
\forall i\in\{1,\ldots,10\},\\
\forall (b,b^*)\in G_j,\,(c,c^*)\in G_k
\end{array}\\
\text{if and only if}\\
f(b,c^*)+\inner{c-b}{b^*}-\inner{c}{c^*}\leq 0,\quad\forall (b,b^*)\in G_j,\,(c,c^*)\in G_k,
\end{gather*}
where $f(b,c^*)=\displaystyle\max_{i=1,\ldots,10} \inner{b-w_i}{w_i^*}+\inner{w_i}{c^*}$. Function $f$ can be implemented in \emph{Maxima} with the following code:
\lstset{language=Maxima}
\begin{lstlisting}
f(opW,b,cs):=block([ww,maxval],
   maxval:-inf,
   for ww in opW do(
      maxval:max(maxval,(b-xc(ww)).xsc(ww)+xc(ww).cs)
   ),
   maxval
)$
\end{lstlisting}

Now consider $b=(x,y)$ and $c^*=(u,v)$. Using the above procedure we obtain
\begin{equation}\label{eq:funf}
f(b,c^*)= \max\left\{
\begin{array}{c}
v-2x,-x+(v+u+1)/2,v-x,u-y,\\
-2(x+1)-u,-2(y+x+1)-u,-y-v-1,\\
-y-x-v-1,-y-x-(v+u)/2-1,y+u\end{array}\right\} 
\end{equation}

To verify the 25 cases of the form $(1,j,k)$, $j,k\in\{2,\ldots,6\}$, consider the following definition in \emph{Mathematica} of the function $f$ as in \eqref{eq:funf},
\lstset{language=Mathematica}
\begin{lstlisting}
fn[x_, y_, u_, v_] := 
   Max[ v - 2 x, -x + (v + u + 1)/2, v - x, -2 (x + 1) - u,
     -2 (y + x + 1) - u, u - y, -y - v - 1, -y - x - v - 1, 
     -y - x - (v + u)/2 - 1, y + u ]
\end{lstlisting}
We next define 
\begin{lstlisting}
hatS = fn[bx, by, csx,csy] 
   + (bsx (cx - bx) + bsy (cy - by)) - (cx csx + cy csy)
\end{lstlisting}
where $b=(bx,by)$, $b^*=(bsx,bsy)$, $c=(cx,cy)$, $c^*=(csx,cxy)$. In order to verify $S\leq 0$, whenever $(a,a^*)\in G_1$ and $(b,b^*)\in G_j$, $(c,c^*)\in G_k$, for  $j,k\in\{2,\ldots,6\}$, we consider the following instruction

\begin{lstlisting}
Table[ 
   Simplify[hatS <= 0, 
      g[bx, by, bsx, bsy] && h[cx, cy, csx, csy]], 
      {g, {G2, G3, G4, G5, G6}}, {h, {G2, G3, G4, G5, G6}}
   ]
\end{lstlisting}
This instruction asks to simplify the claim \lstinline!S<=0! subject to both the conditions 
\lstinline!g[bx, by, bsx, bsy]! and \lstinline!h[cx, cy, csx, csy]! being true. After running the previous instruction, \emph{Mathematica} outputs a list of 25 ``\lstinline!True!'' symbols. 

\paragraph{The case $(i,j,k)$, $i,j,k\neq 1$:}
We now address the remaining 125 cases. In a similar way as before, we define 
\begin{lstlisting}
S = ({bx, by} - {ax, ay}).{asx, asy} 
    + ({cx, cy} - {bx, by}).{bsx, bsy}
       + ({ax, ay} - {cx, cy}).{csx, csy}
\end{lstlisting}
as in \eqref{eq:S}, where $a=(ax,ay)$, $a^*=(asx,asy)$, $b=(bx,by)$, $b^*=(bsx,bsy)$, $c=(cx,cy)$, $c^*=(csx,cxy)$. 
\begin{lstlisting}
Table[{f, g, h, 
    Simplify[S <= 0, 
      f[ax, ay, asx, asy] 
      && g[bx, by, bsx, bsy] && h[cx, cy, csx, csy]]}, 
    {f, {G2, G3, G4, G5, G6}}, 
    {g, {G2, G3, G4, G5, G6}}, {h, {G2, G3, G4, G5, G6}}]
\end{lstlisting}
This instruction creates a $125\times 4$ table with the simplification of the term $S\leq 0$,  alongside the case considered. For instance, the row \lstinline!{G2,G3,G6,True}! means that $S\leq 0$ simplifies to True, whenever $(a,a^*)\in G_2$, $(b,b^*)\in G_3$ and $(c,c^*)\in G_6$.  In this case, however, the array produced contains mostly \lstinline!True! symbols in its fourth column, with the exception of the following rows:
\begin{lstlisting}
 {G2, G5, G6, ax + ax csx + cy + bsy cy <= ay}
 {G5, G6, G2, by + asy by + cx + bsx cx <= cy}
 {G6, G2, G5, ay + bx + asx bx + ay csy <= by}
 {G3, G5, G4, ax + ay + ax csx + bsy cy <= 0}
 {G4, G3, G5, bx + asx bx + by + ay csy <= 0}
 {G5, G4, G3, asy by + cx + bsx cx + cy <= 0}
\end{lstlisting}
These rows correspond to the cases $(i,j,k)=(2,5,6)$ (or, equivalently, $(5,6,2)$ or $(6,2,5)$) and $(i,j,k)=(3,5,4)$ (equivalently, $(5,4,3)$ or $(4,3,5)$). We now deal with these two final cases separately:
\begin{itemize}
 \item When $(i,j,k)=(2,5,6)$, it means that $a=(ax,ay)$, $a^*=(asx,asy)$, $b=(bx,by)$, $b^*=(bsx,bsy)$, $c=(cx,cy)$ and $c^*=(csx,csy)$, satisfy
 \begin{gather*}
0 < ax \leq  ay < 1 - ax,\quad asx = -1,\quad asy = 0,\\
bx = 0 ,\quad by = 0 ,\quad bsx = -1 ,\quad -1 \leq  bsy \leq  0,\\
cx = 0 ,\quad -1 < cy < 0 ,\quad -1 \leq  csx \leq -1 - cy ,\quad csy = -1
 \end{gather*}
Replacing the above equalities on equation~\eqref{eq:S}, we obtain
\[
S= (ax - ay) + (ax \cdot csx) + (1 + bsy) cy.
\]
It trivially follows that $S\leq 0$, since $ax-ay\leq 0$, $ax>0$, $csx<0$, $1+bsy\geq 0$ and $cy<0$.
 \item When $(i,j,k)=(3,5,4)$, it means that $a=(ax,ay)$, $a^*=(asx,asy)$, $b=(bx,by)$, $b^*=(bsx,bsy)$, $c=(cx,cy)$ and $c^*=(csx,csy)$, satisfy
 \begin{gather*}
-1 - ax < ay <= ax < 0,\quad asx = -1,\quad asy = -1,\\
bx = 0 ,\quad by = 0 ,\quad bsx = -1 ,\quad -1 \leq  bsy \leq  0,\\
cx = 0 ,\quad  0 < cy < 1 ,\quad -1-cy \leq  csx \leq -1  ,\quad csy = 0
 \end{gather*}
Replacing the above equalities on equation~\eqref{eq:S}, we obtain
\[
S= ax + ay + ax\cdot csx + bsy\cdot cy\leq 2ax+ax\cdot csx+bsy\cdot cy,
\]
since $ay\leq ax$. It follows that $S\leq 0$, since $ax< 0$, $csx>-2$, $bsy\leq 0$ and $cy>0$.
\end{itemize}

We now are able to conclude that $T_6$ is 2-cyclically monotone, and therefore, maximal 2-cyclically monotone.
